\documentclass{amsart}
\usepackage{graphicx} 

\usepackage{hyperref}

\usepackage{todonotes}
\usepackage{leftindex}

\usepackage[nameinlink]{cleveref}

\usepackage{subcaption}

\newtheorem{theorem}{Theorem}

\newtheorem{proposition}[theorem]{Proposition}

\newtheorem{corollary}[theorem]{Corollary}

\theoremstyle{definition}
\newtheorem{definition}[theorem]{Definition}
\newtheorem{notation}[theorem]{Notation}
\newtheorem{example}[theorem]{Example}

\newtheorem{assumption}[theorem]{Assumption}

\theoremstyle{remark}
\newtheorem{remark}[theorem]{Remark}

\newcommand{\NN}{ {\mathbb N} }
\newcommand{\RR}{ {\mathbb R} }
\newcommand{\CC}{{\mathbb C}}

\newcommand{\cP}{ {\mathcal P} }
\newcommand{\cPP}{{\mathcal{PP}}}

\newcommand{\GG}{{\mathcal G}}
\newcommand{\cV}{{\mathcal V}}

\newcommand{\KK}{{\kappa}}

\newcommand{\UU}{\mathcal{O}}

\newcommand{\ff}{\varphi}

\newcommand{\tr}{\mathrm{tr}}
\newcommand{\Tr}{\mathrm{Tr}}
\newcommand{\EE}{\mathbb{E}}

\newcommand{\cA}{\mathcal{A}}

\newcommand{\Wg}{\mathrm{Wg}}

\newcommand{\cc}{c}

\newcommand{\kk}{\kappa}
\newcommand{\moeb}{\mathrm{\text{M\"ob}}}

\newcommand{\ab}{\allowbreak}

\allowdisplaybreaks

\begin{document}

\title[Functions on rotationally-invariant matrices]{Entrywise application of non-linear functions on orthogonally invariant matrices}
\author{Roland Speicher}
\author{Alexander Wendel}
\address{Saarland University,
Saarbr\"ucken, Germany}
\thanks{We thank Alberto Dayan for discussions which inspired us to look also at the multivariate case.\\
The second author has been supported by the Deutsche Forschungsgemeinschaft DFG, Project SP419/11-1.}

\begin{abstract}
    In this article, we investigate how the entrywise application of a non-linear function to symmetric orthogonally invariant random matrix ensembles alters the spectral distribution. 
    We treat also the multivariate case where we apply multivariate functions to entries of several orthogonally invariant matrices; where even correlations between the matrices are allowed. We find that in all those cases a Gaussian equivalence principle holds, that is, the asymptotic effect of the non-linear function is the same as taking a linear combination of the involved matrices and an additional independent GOE. The ReLU-function in the case of one matrix and the max-function in the case of two matrices provide illustrative examples.
\end{abstract}
\maketitle

\section{Introduction}

We consider random matrix ensembles $X_N=(x_{ij})_{i,j=1}^N$ which have a limiting eigenvalue distribution for $N\to\infty$. In the machine learning context applying non-linear functions $f$ to the entries of such matrices plays a prominent role and the question arises: what is the asymptotic effect of this operation. There are a couple of results in this direction; see below for some references. However, it seems that they deal with quite special choices for $X_N$, like a product of two independent Wishart matrices or similar constructions. In those cases the asymptotic effect of applying the non-linearity is the same as taking a linear combination of the $X_N$ with an independent Gaussian matrix. The coefficients in this linear combination depend only on a few quantities calculated from the non-linear function. Such statements are often known as Gaussian equivalence principle.

It seems to us that this kind of result is also true much more generally, namely for rotationally invariant matrices; and that the rotational invariance is kind of the underlying reason for the specific form of the result. Roughly, the effect of the non-linearity is only a small deformation of the invariance property. The resulting random matrix ensemble exhibits still an asymptotic form of rotational invariance, which gives precise enough information to control the limiting eigenvalue distribution. We want to work out these ideas in the following. Our approach was motivated by the work of Piccolo and Schr\"oder \cite{piccolo2021analysis} on the above mentioned results for special situations.

There are two situations which are usually considered in this context: namely 
\begin{enumerate}
\item
either $X_N$ is a matrix without any symmetry condition (and then actually it is also allowed to be rectangular), in this case one looks on the eigenvalue distribution of $X_NX_N^*$; this setting was considered, for example, by Louart, Liao, and Couillet \cite{louart2018random} and by Pennington and Worah \cite{pennington2017nonlinear}, see also \cite{benigni2021eigenvalue,fan2020spectra};
\item
or $X_N$ is a symmetric matrix and we look directly on its eigenvalues; this case was considered by El Karoui \cite{el2010spectrum} and by Cheng and Singer \cite{cheng2013spectrum}, see also \cite{fan2019spectral,goldt2022gaussian} . 
\end{enumerate}

We will consider here the symmetric case. Then it also seems to make a difference whether the matrices are real or complex (and thus symmetric or selfadjoint); the real case seems actually to be the more interesting one and as it is also more appropriate in statistical or machine learning settings we will only consider this case, that is, symmetric orthogonally invariant random matrices.

\section{Classical and free cumulants for orthogonally invariant random matrices}

We will look on random matrices $X_N=(x_{ij})_{i,j=1}^N$ where the entries are real random variables, satisfying the constraints $x_{ij}=x_{ji}$ for all $i,j=1,\dots,N$. (The entries $x_{ij}$ of $X_N$ depend of course also on $N$, but we will suppress this in the notation.) Orthogonal invariance then means that the joint distribution of the entries of $X_N$ is the same as of the conjugated matrix $O_N X_N O_N^T$, for any orthogonal $N\times N$-matrix $O_N$. One considers random matrix ensembles where the probability distribution on the entries is given in a uniform way for all sizes $N$, such that one has a limit of the relevant quantities. 


Typically, our orthogonally invariant random matrix ensembles are given either by specifying an invariant density of the form
$$\exp\left(-N\Tr(q(X_N))\right)dX_N$$ (for some polynomial, or more general function, $q$) or by looking on non-commutative polynomials in independent GOE matrices. In both those cases one knows that $c_n(\Tr(X_N^{m_1}),\dots,\Tr(X_N^{m_n}))$ is asymptotically of order $N^{2-n}$ and that the limits
\begin{equation}\label{eq:limitsofcorrelations}
\lim_{n\to\infty} N^{n-2} \cdot c_n(\Tr(X_N^{m_1}),\dots,\Tr(X_N^{m_n}))
\end{equation}
exist for all $n$ and all choices of $m_1,\dots,m_n$. The quantities $c_n$ are here the classical cumulants, considered as $n$-linear functionals; i.e., $c_1$ is the expectation $\EE$, $c_2$ is the covariance, etc. 
Given the existence of all limits in \eqref{eq:limitsofcorrelations}, we say that the random matrix ensemble has a \emph{limit distribution of all orders}. In the following, we will only consider orthogonally invariant random matrices with this property.

We will try to understand the effect of the entrywise application of a, typically non-linear, function $f$ on such matrices by investigating its effect on the classical cumulants of the entries of the matrix. Our main interest is of course not the ``microscopic'' information about the entries, but the ``macroscopic'' information given by the quantities in \eqref{eq:limitsofcorrelations}; in particular, we need to calculate the averaged eigenvalue distribution of the matrix, corresponding to the case $n=1$. However, in the case of our orthogonally invariant matrix ensemble with a limit distribution of all orders there is, in the asymptotic regime $N\to\infty$, a very clear relation between the distribution of the matrix and the distribution of its entries, which takes a particular nice form in terms of cumulants. Let us denote as above by $c_n$ the classical cumulants, considered as multi-linear forms. Then we have on the microscopic level the collection of all the cumulants of the entries
$c_n(x_{i_1j_1},x_{i_2j_2},\dots,x_{i_nj_n})$
for all $n$ and all possible choices of the $x_{ij}$. 
Of this huge collection most quantities have asymptotically quite small order. Namely the orthogonal invariance and the asymptotics of the macroscopic moments as in \eqref{eq:limitsofcorrelations} has the effect that the order of an $n$-th order cumulant can be at most $N^{-(n-1)}$ and this maximal order is only achieved for
cumulants with a cyclic index structure, i.e., those which can, by permuting the arguments and possibly also exchanging the two indices at $x_{ij}=x_{ji}$, be brought into the form
$c_n(x_{i_1i_2},x_{i_2i_3},\dots,x_{i_ni_1})$.
Furthermore, the value of such a cyclic cumulant does in leading order not depend on the choice of the indices. Hence the main microscopic information coming from the entries of our matrix ensemble is given by
\begin{equation}\label{eq:def-kappa}
\kappa_n:=\lim_{N\to\infty} N^{n-1}c_n(x_{i_1i_2},x_{i_2i_3},\dots,x_{i_ni_1}).
\end{equation}
The eigenvalue distribution of our matrix is given in terms of the moments
$$\EE[\tr(X_N^n)]=\frac 1N\sum_{i_1,\dots,i_n=1}^N \EE[x_{i_1i_2}x_{i_2i_3}\cdots x_{i_ni_1}].$$
It now turns out that those macroscopic moments are built out of the microscopic $\kappa_n$ via the free moment-cumulant formula; or to put it another way, the $\kappa_n$ (which are defined in \eqref{eq:def-kappa} as the scaled limits of the classical cyclic cumulants of the entries) are nothing but the free cumulants of the limiting eigenvalue distribution.
(For more information and basic properties of classical and of free cumulants, as well as the free moment-cumulant formula, see\cite{nica2006lectures,mingo2017free}.) The relation between the classical cumulants of the entries and the free cumulants of the matrix for unitarily invariant ensembles is addressed, e.g., in \cite{collins2007second} and for orthogonally invariant ensembles the corresponding formula for the free cumulants was given in \cite[Proposition~6.1]{capitaine2007convolutionsII}.)

It is also interesting to note that such random matrices and their relation to free cumulants have also gained some recent interest in the context of basic physical models and questions, in particular in relation to the quantum symmetric simple exclusion process or the Eigenstate Thermalization Hypothesis; for more information on those, see in particular \cite{pappalardi2022eigenstate, bernard2024structured,fava2023designs, cite-key}.

In the following proposition we collect the relevant information about how the cumulants of entries of orthogonally invariant ensembles scale with $N$. This is very similiar to the complex case as in \cite{collins2007second}. Especially the relation between the free cumulants and cumulants of cycles in the orthogonal case has first been proven (maybe somewhat implicit) in \cite{capitaine2007convolutionsII} and has been rediscovered in the theoretical physics and machine learning community in \cite{maillard2019high} and in the complex case in \cite{bernard2024structured}. Since it is hard to localize the general statement for the orthogonal case in precise form in the literature (in particular with respect to the order of cumulants with several disjoint cycles), we will provide in the appendix a proof by adapting the arguments in \cite{collins2007second} from the unitary to the orthogonal case.

Note that we can do some operations on the entries $x_{ij}$ of the classical cumulants $c_n$ without changing the value of the cumulant: namely, the entries can be permuted and, by the symmetry condition of $X_N$, we can also replace an entry $x_{ij}$ by its symmetric version $x_{ji}$. In the following we will freely apply such operations and address this as arranging the indices.

The main ingredient for the order of the cumulants is in the cycle structure of their indices. 
We say that the indices of $(x_{i_1j_1}, \dots, x_{i_nj_n})$ form $r$ cycles (of lengths $l_1,l_2,\dots,l_r$ with $l_1+\cdots+l_r=n$) if they can be arranged such that we have 
\begin{multline*}
j_1 = i_2, \ j_2  = i_3, \  \dots,\  j_{l_1} = i_1,\qquad 
j_{l_1+1}=i_{l_1 +2},\ \dots,\ j_{l_1+l_2}=i_{l_1+1},\quad \dots\\
\\ \dots\quad 
j_{n-l_r+1}=i_{n-l_r+2},\ \dots,\ j_n=i_{n-l_r+1}
\end{multline*}
If we insist that the indices in different cycles have to be different then we say that the cycles are disjoint or non-connected. If some of the indices in two cycles are the same then we say that those two cycles are connected. If one cycle can be broken into smaller cycles (because some of its indices repeat) we say that the cycle has subcycles.

Note that actually the orthogonal invariance implies that conjugation with $F_j = \sum_{i=1, i \neq j}^N E_{ii} - E_{jj} $ leaves the joint distribution invariant. This has the consequence that, in the language of \cite{maillard2019high}, the contributions of Non-Eulerian graphs are exactly zero, not only in the limit $ N \rightarrow \infty$ as \cite{maillard2019high} states. Together with the results from \Cref{section3} we can thus state the following.

\begin{proposition}\label{lemma:orthcumulants}
Let $X_N=(x_{ij})_{i,j=1}^N$ be a symmetric orthogonally invariant random matrix ensemble which has a limit distribution of all orders (that is, all limits as in \Cref{eq:limitsofcorrelations} exist). Let $i_1, \dots i_n, j_1, \dots, j_n \in \left[ N \right]$. Then the following holds for the joint cumulants $c_n(x_{i_1j_1}, \dots, x_{i_nj_n})$.
\begin{enumerate}
        \item Odd cumulants of single off-diagonal entries vanish: $c_{2n+1}(x_{ij}, \dots, x_{ij})=0$ for $i\not=j$. In particular, the off-diagonal entries are centered.
        \item Joint cumulants $c_n(x_{i_1j_1}, \dots, x_{i_nj_n})$ are non-vanishing only if the indices form cycles. The order of such a cumulant is $\mathcal{O}(N^{2-r-n})$ if the indices form $r$ disjoint cycles.
        \item
        In the case of a single cycle (where subcycles are allowed, that is, some or all of the indices can be the same) we have thus:
	\begin{align*}
		c_n(x_{i_1i_2},x_{i_2i_3}, \dots, x_{i_ni_1}) =  \mathcal{O}(N^{1-n}).
    \end{align*}
    In this case, we actually have that
    \begin{align*}
        \lim_{N \rightarrow \infty} N^{n-1} c_n(x_{i_1i_2},x_{i_2i_3}, \dots, x_{i_ni_1}) = \kk_n,
    \end{align*}
    where $\kk_n$ is the $n$-th free cumulant of the limiting spectral distribution of $X_N$.
\end{enumerate}

\end{proposition}
An example of one cycle is
$$c_4(x_{12},x_{23},x_{34},x_{41})\sim N^{-3}$$
whereas
$$
c_4(x_{12},x_{21},x_{34},x_{43})\sim N^{-4}$$
has two cycles.

Note that usually there are also subleading orders, coming from subcycles; those are harder to control, but luckily they are also not surviving in the calculation of the macroscopic moments. For example, $c_4(x_{12},x_{21},x_{14},x_{41})$ behaves asymptotically in leading order like $c_4(x_{12},x_{23},x_{34},x_{41})$, but has additional subleading orders.

\section{Application of non-linear functions}

Now let $f:\RR\to\RR$ be a, in general non-linear, function and let us apply this entrywise and in an appropriate scaling to our matrix ensemble. Thus we define $Y_N=(y_{ij})$ by
\begin{equation}\label{eq:def-yij}
y_{ij}=\begin{cases}
\frac 1{\sqrt{N}} f(\sqrt{N} x_{ij}), & i\not=j\\
0,& i=j.
\end{cases}
\end{equation}
It will be crucial to get rid of the diagonal terms, because otherwise they would dominate the leading order and we would have to scale in a different way.

For convenience, we will also assume the following property of our function $f$.
\begin{assumption}[Centeredness]
    We assume that our function is such that $f$ is centered w.r.t. the distribution of the off-diagonal entries $x_{ij}$ for any $i,j \in [N]$ with $i \neq j$. 
\end{assumption}
    This restriction can be overcome by adding to $Y_N$ a deterministic matrix which may possess a large outlier eigenvalue and we defer the investigation of additive perturbations of our non-linear ensemble $Y_N$ to future research.

Since we can write a polynomial as a linear combination of monomials,  
$$f(x) = \sum_{m=1}^d a_m x^m,$$ 
it suffices, by multilinearity of the cumulants, to consider the effect of applying, potentially different, powers $f_i(x)=x^{m_i}$ to the arguments, that is:
\begin{align*} \label{cum:multilin}
c_n\left[ f(x_{i_1j_1}), f(x_{i_2j_2}), \dots, f(x_{i_nj_n}) \right] = \sum_{m_1, \dots, m_n} a_{m_1} \cdots a_{m_n} c_n\left[ x_{i_1j_1}^{m_1}, x_{i_2j_2}^{m_2}, \dots, x_{i_nj_n}^{m_n} \right],
\end{align*}
and thus, we need to understand cumulants of the form $c_n\left[ x_{i_1j_1}^{m_1}, x_{i_2j_2}^{m_2}, \dots, x_{i_nj_n}^{m_n} \right]$.

The main tool for transferring properties from $X_N$ to $Y_N$ is the formula of Leonov and Shirayev, which allows to deal with cumulants where the arguments are products. We recall this formula in the appendix in \Cref{Leonov}. In our present setting, the partition $\tau$ from that formula is
$$\tau=\{\{1,\dots,m_1\},\{m_1+1,\dots,m_1+m_2\},\dots,\{m_1+\dots+m_{n-1}+1,\dots,m_1+\dots +m_n\}\}$$

\begin{proposition}\label{lemma:cyclicycumulants}
    Let $f(x) = \sum_{k=1}^d a_k x^k $ be a polynomial of degree $d$. Then, the following holds for the above nonlinear random matrix:
    \begin{enumerate}
    \item 
    All first cumulants $c_1(y_{ij})$ are equal to zero.
        \item For $n \geq 2$, the joint cumulants of the entries of $Y_N$, $c_n(y_{i_1j_1}, \dots, y_{i_nj_n})$ with $i_1, \dots i_n, j_1, \dots, j_n \in \left[ N \right]$, such that $i_k \neq j_k, \forall k = 1, \dots, n$ are non-vanishing only if the indices form cycles.
    \end{enumerate}
\end{proposition}

\begin{proof}
    The first cumulants of the $y_{ij}$ are all zero, either by the definition $y_{ii}=0$ or by the assumption on our polynomial.

    The cyclicity of the joint cumulants of $Y_N$ is a direct consequence of the formula of Leonov and Shiryaev. A partition $\pi$ that contributes to $c_n$ has to satisfy $\pi \vee \tau = 1$ and thus has to connect the different blocks of $\tau$. However, since we set the diagonal entries to zero, we have that $i_k \neq j_k$ or otherwise the corresponding cumulant will be zero. Since the random matrix $X_N$ is orthogonally invariant, we know, by \Cref{lemma:orthcumulants},that the blocks of $\pi$ need to connect indices that form cycles. And in case the indices form a full cycle we know that $\pi$ has to contain at least one such full cycle in order to give a non-zero contribution.

    Now, let us check what we get if the index set $i_1,i_2, \dots, i_n$ has subcycles, that is indices repeat. In this case, by the formula of Leonov and Shiryaev for $c_n(y_{i_1i_2}, \dots, y_{i_ni_1} )$, we again need to ensure, that $\pi \vee \tau = 1$, that is if we follow the connections between blocks, we can move from any block into any other block. But if now, we only want to have $\pi$ connecting the indices within each subcycle, then the former condition cannot be satisfied: If $\pi$ contained blocks, such that each block only connects indices forming strict subcycles, then $\pi$ does not fulfill the requirements of the formula of Leonov and Shiryaev. In consequence, there has to be at least one block of $\pi$ connecting one or several of these cycles. But now, by cyclity of the cumulants of $X_N$, we know that the indices of both blocks have to form one connected cycle (with subcycles). Inductively, we thus see that the partition $\pi$ has to contain at least one full cycle of length $n$. Now, if we have several disconnected cycles (that are allowed to possess subcycles themselves), then by the same arguments as before the partition $\pi$ has to connect the blocks corresponding to the disjoint cycles, which can be satisfied e.g. by a choosing one cycle from each the disjoint cycles. Note that it might also be possible to choose one element of the cycle twice if the degrees allow for this. These contributions will be subleading, see below for a more refined analysis. Finally, if none of these cases occurs, then by orthogonal invariance the corresponding cumulant will be zero.
\end{proof}

Now we will discuss the orders of the cumulants and their asymptotic contribution in the calculation of macroscopic traces.

\begin{theorem}\label{thm:cumulants}
  Let $X_N$ be a symmetric orthogonally invariant random matrix ensemble which has a limit distribution of all orders (that is, all limits as in \Cref{eq:limitsofcorrelations} exist).
   Let, for a fixed choice of a polynomial $f$, $Y_N$ be defined as in \Cref{eq:def-yij}.
 Then the classical cumulants 
$c_n(y_{i_1j_1},y_{i_2j_2},\dots,y_{i_nj_n})$
of the entries $y_{ij}$ of the random matrix $Y_N=(y_{ij})_{i,j=1}^N$ have the following properties.

    \begin{enumerate}
        \item The only cumulants which make asymptotically a contribution in the calculation of the moments of $\EE[\tr(Y_N^n)]$ are those with a cyclic index structure, that is those of the form $c_n(y_{i_1i_2},y_{i_2i_3},\dots,y_{i_ni_1}$); also the contribution of such cyclic indices, for which some of the indices in the cycle are the same, will vanish; thus we only have to deal with cyclic indices where all the indices in the cycle are different; for those the cumulant does not depend on the actual value of the indices and thus we can define the limits
    $$\kk_n^f:=\lim_{N\to\infty} N^{n-1}c_n(y_{i_1i_2},y_{i_2i_3},\dots,y_{i_ni_1})\qquad
    \text{all $i_k$ distinct}.$$
    Those $\kk_n^f$ are thus the free cumulants of the asymptotic eigenvalue distribution of the random matrix $Y_N$.
    \item We have the following relation between the free cumulants $\kk_n$ of the asymptotic eigenvalue distribution of $X_N$ and the free cumulants $\kk_n^f$ of the asymptotic eigenvalue distribution of the random matrix $Y_N$:
    \begin{enumerate}
    \item  
    $\kk_1^f=0$;
    \item
    $\kk_2^f=c_2(f(g),f(g))=\EE[f(g)^2]-\EE[f(g)]^2$;
    \item for $n\geq 3$, $$\kk_n^f=\kk_n\cdot \EE[f'(g)]^n;$$
    \end{enumerate}
where $g$ is a Gaussian random variable with mean zero and variance $\kk_2$.
\end{enumerate}
\end{theorem}

\begin{proof}
Let us first control the leading order of our cumulants $c_n(y_{i_1i_2},y_{i_2i_3},\dots,y_{i_ni_1})$ with cyclic index structure  for monomial activations $f_1(x) = x^{m_1}, \dots, f_n(x) = x^{m_n}$, by again invoking the formula of Leonov and Shiryaev, \Cref{Leonov}. 
\begin{align*}
N^{-n/2} &c_n\left[f_1(\sqrt N x_{i_1i_2}),f_2(\sqrt N x_{i_2i_3}),\dots,f_n(\sqrt N x_{i_ni_1})\right]\\
&=N^{-n/2} c_n\left[(\sqrt N x_{i_1i_2})^{m_1},(\sqrt N x_{i_2i_3})^{m_2},\dots,(\sqrt N x_{i_ni_1})^{m_n}\right]\\
&=N^{(m_1+m_2+\cdots+m_n-n)/2} c_n\left[x_{i_1i_2}^{m_1},x_{i_2i_3}^{m_2},\dots,x_{i_ni_1}^{m_n}\right]\\
&= N^{(m -n)/2}c_n \left[x_{i_1i_2}^{m_1}, x_{i_2i_3}^{m_2}, \dots, x_{i_ni_1}^{m_n}\right] \\ &= N^{(m -n)/2}\sum_{\pi \in \mathcal{P}(m), \pi \vee \tau = 1_{m}} c_{\pi}\left[ x_{i_1i_2}, \dots , x_{i_1i_2}, x_{i_2i_3}, \dots ,  x_{i_ni_1}, \dots, x_{i_ni_1}
\right],   
\end{align*}
where $m = \sum_{i=1}^n m_i$ and each $x_{i_li_{l+1}}$ shows up $m_l$ times in $c_\pi\left[ \dots \right]$.

In the following we will first restrict to the case where all the $i_k$ are distinct.
In this case, since the indices do not have subcycles, only full cycles $(i_1, \dots, i_n)$ of length $n$ can make a contribution and we need indeed at least one such full cycle to connect elements from the different groups of $\tau$. Assume now, that in the above formula, the partition $\pi$ consists of $r$ full cycles of length $n$ and $s_j, j = 2, \dots, n-1$ blocks of size $j$ within each group of $\tau$. Then,
\begin{align*}
    rn + \sum_{j = 2}^{n-1}s_j j = m,
\end{align*}
and the corresponding cumulant $c_{\pi}$ will be of order $\mathcal{O}(N^t)$, where
\begin{align*}
    t=r(1-n) + \sum_{j = 2}^{n-1} s_j(1 - j) = {r + \sum_{j=2}^{n-1} s_j - m} = {\sum_{j=2}^n s_j - m},
\end{align*}
with the convention that $r = s_n$.

Since $m = m_1 + \dots + m_n$ is fixed, we see that we need to maximize the number of blocks of $\pi$ in order to obtain a leading order contribution. This in turn means that $r$ should be as small as possible, i.e., we have $r=1$ (which is necessary since we need at least one such full cycle to satisfy the condition $\pi\vee\tau=1_m$ in the formula of Leonov and Shiryaev); and the rest are all pairings since, by \Cref{lemma:orthcumulants}, $c_1(x_{ij}) = 0$ for $i\not=j$, thus blocks of size 1 make no contribution and only blocks of size $\geq 2$ are relevant. From this we conclude that the leading order comes from partitions $\pi$ with $\frac{m-n}{2}$ blocks of size 2 and exactly one cycle of length $n$. For those the total leading order is again $\mathcal{O}(N^{1-n})$: $c_\pi$ will be of order $\mathcal{O}(N^{\frac{m - n}{2} +1 - m} )$, and including the above prefactor of $N^{(m-n)/2} $ one obtains the order
$\mathcal{O}(N^{1-n})$.

Let us now take a closer look on second cumulants.
\begin{align*}
    \kk_2^{f_1,f_2}
    &=\lim_{N\to\infty} N\cdot N^{(m_1+m_2-2)/2} c_2(x_{12}^{m_1},x_{21}^{m_2}) \\
    &=\lim_{N\to\infty} N^{(m_1+m_2)/2} c_2(x_{12}^{m_1},x_{21}^{m_2})  
\end{align*}
By the above results the leading order case is to have all blocks of size 2 and we can connect each $x$-argument with each other $x$-argument, because $x_{12}=x_{21}$, thus 
$$c_2(x_{12},x_{21})=c_2(x_{12},x_{21})=c_2(x_{21},x_{12})=c_2(x_{21},x_{21}).$$
Each of them needs a factor $N$ to make a limiting multiplicative contribution $\kk_2$. The only constraint from the formula of Leonov and Shiryaev is that the group of the first $m_1$ indices must be connected by some pair with the group of the last $m_2$ indices. If $m_1$ and $m_2$ are odd then this is automatic, if both are even, then we have to subtract from all pairings the ones which connect only within the two subgroups. If the parities of $m_1$ and $m_2$ are different, then there are no pairings at all. Since the number of pairings counts the moments of a Gaussian variable, we get in this case that
\begin{align*}
    \kk_2^{f_1,f_2} &= \lim_{N \rightarrow \infty }  N^{(m_1+m_2)/2} c_2(x_{12}^{m_1},x_{21}^{m_2})\\
    &= \kk_2^{(m_1+m_2)/2}(\mathcal{P}_2(m_1 + m_2) - (\mathcal{P}_2(m_1) \times \mathcal{P}_2(m_2))) \\ &= \mathbb{E}\left[ g^{m_1+m_2}\right] - \mathbb{E}\left[ g^{m_1}\right] \mathbb{E}\left[ g^{m_2}\right]\\
    &=c_2(f_1(g),f_2(g)),
\end{align*}
where $g$ is a Gaussian random variable with mean zero and variance $\kk_2$.
Hence, by the multilinearity of the cumulants, we also get for our polynomial $f$ that
$$\kk_2^f=c_2(f(g),f(g)).$$

Consider now the case $n\geq 3$. Then we have for our monomials
$$\kk_n^{f_1,\dots,f_n}= \lim_{N \rightarrow \infty}N^{n-1}\cdot N^{(m_1+m_2+\cdots+m_n-n)/2} c_n\left[x_{12}^{m_1},x_{23}^{m_2},\dots,x_{n1}^{m_n}\right].
$$
As noted above in order to satisfy the connecting condition we need at least one cyclic block of length $n$ which contains one element from each group of $\tau$ (and which eats up a factor $N^{n-1}$ to produce a $\kk_n$); recall that for this argument we use that all our indices are distinct. The leading order is then given by pairing the remaining elements, by orthogonal invariance necessarily each of them 
within one of the groups. Each such pair needs asymptotically a factor $N$ to produce a contribution $\kk_2$. The cyclic block uses $n$ elements, so we still have to connect $m_1+\cdots m_n-n$ indices, and doing this with pairings will exactly match the remaining  $N^{(m_1+m_2+\cdots+m_n-n)/2}$. So it remains to count the number of possibilities to choose one cyclic block and pairs for the rest.

The cyclic block has $m_1$ possibilities to choose its element from the first group, $m_2$ possibilities to choose its element from the second block and so on, thus there are $m_1\cdot m_2\cdots m_n$ possibilities for choosing a cyclic block. If this is chosen, then the $i$-th group contains $m_i-1$ elements which have to be paired among themselves. Counting the number of those pairings, and weighting each such pair with a factor $\frac{\kk_2}{N}$ gives in the limit $N \rightarrow \infty$ again the same as the $(m_i-1)$-th moment of a Gauss variable of variance $\kk_2$, thus we get:
$$\kk_n^{f_1,\dots,f_n}=\kk_n\cdot m_1 \kk_2 \EE[g^{m_1-1}]\cdots m_n \kk_2 \EE[g^{m_n-1}]=
\kk_n \EE[f_1'(g)]\cdots \EE[f_n'(g)]
;$$
and so, by multilinearity, for our polynomial $f$
$$\kk_n^f=\kk_n\cdot \EE[f'(g)]^n.$$

It remains to deal with the cases where we have 
subcycles or several disconnected cycles in the index. 
In those cases the application of the non-linearity might have a non-trivial effect on the leading order and the situation is getting more complicated. Nevertheless, we will see that those terms will still not make a contribution to the asymptotic calculation of the moments of the matrix $Y_N$.
Before diving deeper into this analysis, we need some preparations. 
\end{proof}

Let us borrow some definitions from \cite{anderson2010introduction} in order to organize the following arguments.

\begin{definition}[$\mathcal{S}$-word]
Given a set $\mathcal{S}$, an \emph{$\mathcal{S}$-letter} $s$ is simply an element of $\mathcal{S}$. An \emph{$\mathcal{S}$-word} $ w $ is a finite sequence of letters $s_1 \cdots s_n, n \in \mathbb{N}$, at least one letter long. An $\mathcal{S}$-word $w$ is \emph{closed} if its first and last letters are the same. Two $\mathcal{S}$-words $w_1, w_2$  are called \emph{equivalent}, denoted $w_1 \sim w_2$, if there is a bijection on $\mathcal{S}$ that maps one into the other.
\end{definition}

\begin{remark}
    We will in the following interpret our multi-indices $i = (i_1, \dots, i_n)$ as $\mathcal{S}$- words for the set $\mathcal{S} = [ N ]$, or as \cite{anderson2010introduction} calls them, $N$-words.
\end{remark}

\begin{definition}[length, weight, support of a word]
    Let $w = s_1 \cdots s_k$ be an $\mathcal{S}$-word for a set $\mathcal{S}$ with letters $s_i, i =1, \dots, k$. Then, we call $\ell(w) := k$ its \emph{length}; and the \emph{weight} $wt(w)$ of $w$ is the number of different constituent letters, which is the cardinality of the set $\{ s_1, \dots, s_k\}$, . Furthermore, we define the \emph{support} of $w$, $\text{supp}(w)$, to be the set of letters appearing in $w$.
\end{definition}

\begin{definition}[Graph associated with an $\mathcal{S}$-word] \label{def:graph}
    Given a word $ w = s_1 \cdots s_k $, we let $ G_w = (V_w, E_w)$ be the graph with set of vertices $V_w = \text{supp}(w) $ and (undirected) edges $ E_w = \{ \{s_i, s_{i+1}\}, i = 1, \ldots, k-1 \} $. We define the set of self edges as $ E^s_w = \{ e \in E_w : e = \{u, u\}, u \in V_w \} $ and the set of connecting edges as $ E^c_w = E_w \setminus E^s_w $.
\end{definition}

This terminology is specially adapted for the computation of expectations of traces of powers of our matrices. However, this tool is also very suitable for our purpose of describing cyclic cumulants. As \cite{anderson2010introduction} remarks such a graph is connected and by \Cref{lemma:cyclicycumulants} we will only look at graphs which consist of cycles. Furthermore the word/multi-index $w$ induces a path/walk on $G_w$ in the obvious way, which we denote by $p_w$. Let us now turn to the case where the graph $G_w/G_i$ isn't a full cycle anymore.

\begin{proof}[Continuation of Proof]

In the case, where the multi-index $i$ contains several connected cycles, the order of magnitude in $N$ of the cumulant may change when applying a nonlinear transformation as specified above. This is the case if the sets $\{ i_k, i_{k+1} \}$ and $\{ i_l, i_{l+1} \}$ coincide for some $l \neq k$, which can be translated to the property that the loops in the walk $p_i$ on $G_i$ induced by $i$ share one or more edges. Asymptotically, these contributions will be negligible as we will see below, however let us investigate what their leading order looks like. Assuming that the graph $G_i$ consists of $c$ many cycles $c_j, j = 1, \dots, c$ and suppose furthermore that in the walk $p_i$  the edge $e \in E_w^c$ appears $N_e^{i}$ times. By the formula of Leonov and Shiryaev we first have to ensure that the connection between the groups of $\tau$ allow one to reach any other group following the blocks of the partition $\pi$. To ensure that the partition $\pi$ contributes to the formula of Leonov and Shiryaev we first fix for any of the different loops in $p_i$ one element, each from different groups of $\tau$ by the cyclicity of the cumulants of $X_N$ this connection has to be such that the edges of the connected blocks form cycles (and in the leading order exactly one, since otherwise we could split the cycles giving a larger order) and each constituent cycle of the graph has to show up at least once. If the condition $\pi \vee \tau = 1$ isn't already achieved this can be satisfied by choosing connections for the blocks corresponding to the same edge $e$ for example by a block which contains exactly one element from each block. (Note that this requires the degrees to be large enough, if this is not the case, then one can of course also try to put all elements into one block and \Cref{lemma:orthcumulants} applies.) It is clear that this construction can lead to partitions $\pi$ which are crossing, see below for the example if a loop appears several times. The remaining points can now be connected arbitrarily as long as they don't give a factor of zero.

However, since we are interested in the leading order of the cumulant we want to have as much pairings as above, which this construction clearly doesn't yield. Assuming that all powers $m_j$ are large enough (w.r.t.~$n$), we may connect the groups of $\tau$ by choosing each of the different subcycles once in the way we just specified and then pair the remaining points if possible. This shows that the order of the cumulant will be at most
\begin{align*}
    N^{(m-n)/2} \cdot N^{d - \sum_j d_j - (m -\sum_j d_j )/2} = N^{d - \sum_j d_j/2 - n/2},
\end{align*}
where $d$ denotes the number of \emph{distinct} cycles with respective length $d_j, j = 1, \dots , d$, such that this term can be as large as $N^{-n/2}$ (see \Cref{ex:cyclepairing}). This estimate also holds if the $m_j$ are not large enough and in this case this order will be even smaller. Furthermore it is clear from our above considerations and \Cref{lemma:orthcumulants}, that if $\pi$ contains a block which connects two or more disconnected cycles the contribution of $c_{\pi}$ the order of magnitude will be even smaller: If $\pi'$ was another partition, such that we split the blocks of $\pi$ in question into two or more different blocks according to the disjoint cycles, then the order will increase and we can apply our estimate again. We give a more detailed analysis of the case when disjoint cycles are being connected in \Cref{section4}.

If the subcycles are all different and don't share edges we still need a full cycle in $\pi$ as we saw in \Cref{lemma:cyclicycumulants} and will still have the possibility to choose in total $m_1 \cdots m_n$ points for it within each block. Suppose there were $s$ many different such pairs, such that in total this index pair shows up $r_i, i = 1, \dots , s$ many times and denote $m'_i = r_i -1 $. After choosing the elements of the full cycle there will remain for each $\{ i_k, i_{k+1} \}$-pair $\mathcal{P}_2(m'_i)$ many possibilities for a pairing, such that in the limit $N \rightarrow \infty$ we will obtain $$\lim_{N\rightarrow \infty} N^{n-1} c_n(x_{i_1i_2}, \dots, x_{i_ni_1} ) = \kk_n \cdot m_1 \cdots m_n \prod_{i=1}^s \mathcal{P}_2(m'_i).$$ The resulting limit will thus be the same as if we were computing the joint cumulant $c_n(x_{i_1i_2} g_{i_1i_2}^{m_1-1}, x_{i_2i_3}g_{i_2i_3}^{m_2-1} , \dots, x_{i_ni_1} g_{i_ni_1}^{m_n-1} )$, where $G = (g_{ij})$ is a GOE matrix independent from $X_N$, where each off-diagonal entry has variance $\frac{\kk_2}{N}$. Careful analysis also reveals that the leading order can be interpreted similarly if the cycles repeat and/or share edges despite the fact now the leading order is larger than $\mathcal{O}(N^{1-n})$.

Let us now, finally, check that indeed the quantities $\kk_n^f$ define the free cumulants of the limiting spectral distribution of $Y_N$. For this consider the computation of an $n$-th moment of $Y_N$:
\begin{align*}  
    \EE[\tr(Y_N^n)]=\frac 1N\sum_{i_1,\dots,i_n=1}^N \EE[y_{i_1i_2}y_{i_2i_3}\cdots y_{i_ni_1}].
\end{align*}
Now, we can apply the moment cumulant-formula for the entries of $Y_N$ and arrange the terms in the sum according to how many blocks the partition $\pi$ has:
\begin{align*}
    \frac{1}{N}\sum_{i_1,\dots,i_n=1}^N \EE[y_{i_1i_2}&y_{i_2i_3}\cdots y_{i_ni_1}]\\ &= \frac{1}{N}\sum_{i_1,\dots,i_n=1}^N\ \sum_{\pi \in \mathcal{P}(n)}c_\pi \left[y_{i_1i_2},y_{i_2i_3}, \dots, y_{i_ni_1} \right]\\
    &= \frac{1}{N}\sum_{i_1,\dots,i_n=1}^N\ \sum_{k = 1}^n \ \sum_{\substack{\pi \in \mathcal{P}(n)\\ \#\pi = k}} c_\pi \left[y_{i_1i_2},y_{i_2i_3}, \dots, y_{i_ni_1} \right] \\
     &= \frac{1}{N}\sum_{k = 1}^n\  \sum_{\substack{\pi \in \mathcal{P}(n)\\ \#\pi = k}} \ \sum_{i_1,\dots,i_n=1}^N c_\pi \left[y_{i_1i_2},y_{i_2i_3}, \dots, y_{i_ni_1} \right].
\end{align*}

The above results require the blocks of $\pi$ to connect the $y_{i_{k}i_{k+1}} $ in such a way that the indices of the connected elements form several cycles. Further we know, that the leading order will be the decomposition of the multiindex into its maximum number of cycles. This follows essentially from the same argument as above, and we saw that this corresponds to those situations in which a partition $\pi$ in the formula of Leonov and Shiryaev has the maximal number of blocks. In this case each subcycle has to appear once: $\#\pi = \# (\text{cycles in } i)$ and the above decomposition actually is the decomposition (at least asymptotically) into the number of cycles of the indices. A very similar observation had already been made for orthogonally invariant random matrices in \cite{maillard2019high}; note however, that our $Y_N$ are not orthogonally invariant, see also the remark below.

Now, we can depict a contribution of a multi-index graphically as we did in \Cref{def:graph}. Say we are interested in computing the $n$-th moment and the multi-index $i = (i_1,i_2, \dots, i_n)$ takes $wt(i)$ many different values. Then, we can depict this contribution by $G_i$ and the walk $p_i$ on it. The graph will be connected \cite{anderson2010introduction} and consists of cycles, which are glued together at some of its vertices . A contribution $c_\pi$ with $\pi$ possessing $k$ many blocks corresponds to a decomposition of this graph into $k$ loops/cycles (note the abuse of nomenclature) since we know that only cyclic cumulants are non-zero for $Y_N$. 

Suppose now, that a sequence $i_{k}, \dots, i_{k+l}  $ in the walk on the graph $G$ appears twice, then clearly a noncrossing contribution of $\pi$ may occur. However, these contributions will be subleading, since the double appearance of indices reduces the number of total vertices of the graph $G$. For this let us further decompose the set of multiindices $i = (1_1, \dots, i_n)$ according to the corresponding graph $G_i$:
\begin{align*}
 \sum_{i_1,\dots,i_n=1}^N c_\pi \left[y_{i_1i_2},y_{i_2i_3}, \dots, y_{i_ni_1} \right] &= \sum_{G} \sum_{\substack{i_1,\dots,i_n=1\\ G = G_i}}^N c_\pi \left[y_{i_1i_2},y_{i_2i_3}, \dots, y_{i_ni_1} \right].
\end{align*}

Assume now, that the partition $\pi$ which appears in the above formula is crossing. This implies, that in the walks $p_i$ corresponding to the contribution $c_\pi(y_{i_1i_2}, \dots, y_{i_ni_1})$ at least one edge must repeat or that several disjoint cycles are being connected by $\pi$. Assuming the first case and without loss of generality that the degrees $m_j$ are large enough, using the conventions from above, we know that the individual contribution of an index in this setting is at most of order $N^{d - \sum_jd_j/2 -n/2}$, where $d, d_j$ depend on $i$. But actually this contribution only depends on the graph $G_i$ and its associated walk $p_i$ by the orthogonal invariance of $X_N$ and the resulting exchangeability of the entries of $Y_N$. There are at most $N^v = N^{\sum_j d_j - d}$ many indices contributing for fixed $G_i/p_i$, where $v$ denotes the number of vertices of $G$. Thus, also including the prefactor ${1}/{N}$ from the normalized trace, and convincing oneself of the obvious fact that $\sum_j d_j \leq n$ we conclude that in the case of repeating edges the total contribution will be of order $\mathcal{O}(N^{-1})$. In case that $\pi$ connects several disjoint cycles, then we already noted that splitting these cycles will increase the order by a factor of $N$ for each 'cut', i.e. the order is a factor $N$ lower. In case $\pi$ is still crossing after removing all such connections we can infer that that this contribution will be subleading. If the resulting partition is noncrossing, then we can deduce that these contributions are also vanishing in the limit $N \rightarrow \infty$ (they are $o_N(N^{1-n})$) by the following discussion together with the above consideration of splitting blocks.

If $\pi$ is noncrossing, let us first focus on the case where the decomposition of $\pi$ into blocks also corresponds to a decomposition of $p_i/ G_i$ into its disjoint cycles. In this case the graph $G$ will be a so-called 'cactus-graph', see \cite{maillard2019high}. Suppose the blocks $V_j$ of $\pi$ have sizes $s_j$, respectively, then $c_\pi(y_{i_1i_2}, \dots, y_{i_ni_1})$ will be of order $N^{k - \sum_j s_j} = N^{k - n}$. Now, there will be $N^{n - k +1}$ many indices contributing, which shows that if $\pi$ connects several disjoint cycles, the associated contribution will be subleading. Furthermore if a block of $\pi$ contained a cycle with subcycles, then again by splitting this block into several blocks with one cycle and our previous considerations we can infer that such contributions are subleading. From this we can conclude that in the limit we obtain:
\begin{align*}
    \lim_{N \rightarrow \infty} \mathbb{E}[\tr (Y_N)^n ] = \sum_{\pi \in NC(n)} \kappa^f_\pi.
\end{align*}
This is precisely the free moment cumulant formula, which finishes the proof.
\end{proof}

\begin{remark}[Noncrossing partitions and cactus graphs]
    Note that a cactus graph can be naturally identified with a noncrossing partition: The cactus graph is, as \cite{maillard2019high} states, a tree made of cycles. The tree structure together with the relative position of the cycles then naturally defines the nesting structure of a noncrossing partition and the length of the cycles the block sizes. In \cite{maillard2019high} it has been observed that for orthogonally invariant ensembles contributions to $\mathbb{E}[\sum_{i, G_i = G} x_i] $ associated to a cactus graph $G$ can be factorized into free cumulants, where the sum runs over all multiindices $i \in [N]^n$, such that for fixed $G$ we have $G = G_i$. This is essentially nothing but the free moment-cumulant formula for orthogonally invariant random matrice; this interpretation of cactus graphs as noncrossing partitions was not spelled out explicitly in \cite{maillard2019high}. Note that \cite{maillard2019high} computes these quantities for $X_N = O D O^T$ as expectation over the orthogonal group and obtains results a.s. with respect to the law of $D$.
\end{remark}

\begin{example}\label{ex:cyclepairing}
Let us here give one example of the situation, where an index contains repeated indices (and thus subcycles), in order to show that indeed, the order of magnitude of a cumulant may attain $N^{n/2}$.

    Consider, for example, the cumulant
    $c_4(y_{12},y_{21},y_{12},y_{21})$.
    This corresponds to a cycle of length four in the index, which has two subcycles of size two each. In the case of our orthogonally invariant matrix $X_N$ the corresponding cumulant
    $c_4(x_{12},x_{21},x_{12},x_{21})$ has leading order $N^{-3}$, corresponding to the cycle of length 4, and a subleading contribution of order $N^{-4}$; thus the leading order is the same as for the situation  $c_4(x_{12},x_{23},x_{34},x_{41})$.
    For the non-linear matrix $Y_N$, however, the repetition of indices makes a difference and having subcycles can actually increase the order compared to the case without subcycles. For concreteness, consider $f(x)=x^3$.
    Then, in the case without subcycles 
    $$c_4(y_{12},y_{23},y_{34},y_{41})
    =N^4 c_4(x_{12} x_{12} x_{12}, x_{23} x_{23} x_{23}, x_{34} x_{34}x_{34}, x_{41} x_{41}x_{41} )$$
    has leading order $N^4 N^{-3}N^{-4}=N^{-3}$ 
    (corresponding to a $\pi$ with one block of length 4 and four blocks of length 2), which is the same order as for $X_N$. However, in the case with subcycles
    $$c_4(y_{12},y_{21},y_{12},y_{21})
    =N^4 c_4(x_{12} x_{12} x_{12}, x_{21}x_{21} x_{21}, x_{12} x_{12}x_{12}, x_{21} x_{21}x_{21} )$$
    has leading order $N^4N^{-6}=N^{-2}$ (corresponding to six blocks of length 2). Although the individual contribution of such a cumulant is bigger than the contribution of the case without repeated indices, the overall effect of such subcycles is negligible in the limit. Whereas in the case without subcycles we have $N^4$ many indices, each of which contributes with order $N^{-3}$, resulting (by also taking into account the factor $1/N$ from the normalized trace) in an asymptotic contribution to $E[\tr(Y_N^4)]$, in the case of subcycles we have only $N^2$ many indices, each of which contributes with an order $N^{-2}$ and thus making no contribution to $E[\tr(Y_N^4)]$ in the limit.

    Note that the cumulant $c_4(y_{12},y_{21},y_{12},y_{21})$ presents the worst situation, since here we can connect the two subcycles by pairings, which results then in the increase of the order compared to the orthogonally invariant case. Consider 
    $c_4(y_{12},y_{21},y_{13},y_{31})$ instead. There we have also two subcycles, but they are only connected (in the graph language) at an vertex, but don't share a common edge. This means that we cannot connect them via pairings, but still need a 4-cycle, or two disconnected 2-cycles for doing so. The first case gives the same order as 
    $c_4(y_{12},y_{23},y_{34},y_{41})\sim N^{-3}$,
    whereas the second case gives the same order as
    $c_4(y_{12},y_{21},y_{34},y_{43})\sim N^{-4}$; 
    thus in this situation there is no increase in the order compared to the orthogonally invariant situation.
    
\end{example}

\section{Gaussian equivalence principle}
We have seen that the free cumulants of the asymptotic eigenvalue distribution of $Y_N$ are given in terms of the ones of the original model $X_N$. Actually, it is very easy to produce a linear model $\hat Y_N$ which has the same asymptotic cumulants as the ones from \Cref{thm:cumulants}, and thus also the same asymptotic eigenvalue distribution as $Y_N$.

\begin{corollary}\label{cor:gaussian-equivalence}
Let $X_N$ and $Y_N$ be as above. Then the non-linear random matrix model $Y_N$ has the same asymptotic eigenvalue distribution as the linear model
$$\hat Y_N:=\theta_1 X_N+\theta_2 Z_N,$$
where $Z_N$ is a symmetric standard \text{\rm GOE} random matrix which is independent from $X_N$ and where
$$\theta_1:=\EE[f'(g)]$$
and
\begin{align*}
\theta_2:&= \sqrt{ \EE[f(g)^2]-\EE[f(g)]^2-\kk_2 \EE[f'(g)]^2}\\
&=
\sqrt{\EE[f(g)^2]-\EE[f(g)]^2-\frac 1{\kk_2}\EE[g \cdot f(g)]^2}
\end{align*}
where $g$ is a Gaussian random variable with mean zero and variance $\kk_2$. 
\end{corollary}

In the reformulation of $\theta_2$ we have used Stein's identity for Gaussian integration. Note also that we have
$$\EE[g\cdot f(g)]=\EE[g\cdot (f(g)-\EE[f(g)])],$$ 
since $g$ is centered
and thus, by Cauchy-Schwartz, the quantity under the square root is always positive.

\begin{proof}
Since $Z_N$ is asymptotically free from $X_N$, we get in the large $N$-limit the free cumulants $\hat \kk_n$ of $\hat Y_N$ as the sum of the corresponding free cumulant of $\theta_1 X_N$ and $\theta_2 Z_N$. Since $Z_N$ has a standard semicircle distribution in the limit, i.e., only its second free cumulant is 1, all others vanish, this means
$$\hat \kk_n=\theta_1^n \kk_n+\theta_2^2 \delta_{n2}.$$
It is easy to check that this reproduces the $\kk_n^f$ from \Cref{thm:cumulants}.
\end{proof}

This means that the asymptotic eigenvalue distribution of $Y_N$ is given as the free convolution of a scaled distribution of $X_N$ and a semicircle. Note that in the case $\theta_1=0$ the original information about $Y_N$ dissapears and we just get a semicircle. This happens, for example, if $f$ is an even (and $f'$ thus an odd) function.

\section{Example for the ReLU function}
It is feasible that the statements of \Cref{thm:cumulants} and 
\Cref{cor:gaussian-equivalence} are also true for more general functions than polynomials. We leave the theoretical investigation of this for future investigations, here we only want to check this numerically for the practically relevant ReLU function, i.e.,
$$f(x)=\text{ReLU}(x):=\max(0,x).$$

Typical examples for rotationally invariant matrix ensembles are given by polynomials in several independent GOEs. We take here
\begin{equation}\label{eq:X_Neins}
X_N=A_N^2+A_N+B_NA_N+A_NB_N+B_N,
\end{equation}
where $A_N$ and $B_N$ are independent standard GOE. Figure 1(A) shows the histogram of the eigenvalues of one realization of $X_N$ for $N=5000$.

\begin{figure}\label{fig:1A}
    \centering
    \begin{subfigure}[t]{0.4\textwidth}
        \centering
        \includegraphics[width=5cm]{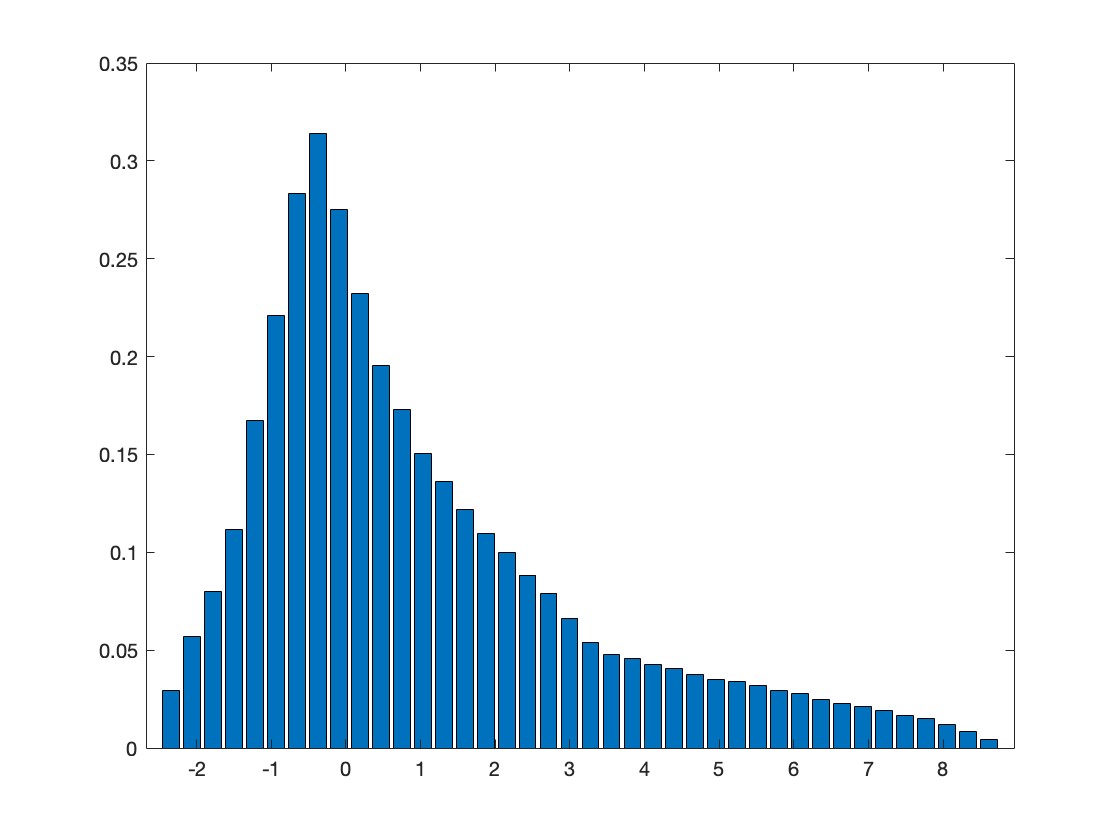}
        \caption{Eigenvalue distribution of $X_N$}
    \end{subfigure}%
    ~ 
    \begin{subfigure}[t]{0.4\textwidth}
        \centering
        \includegraphics[width=5cm]{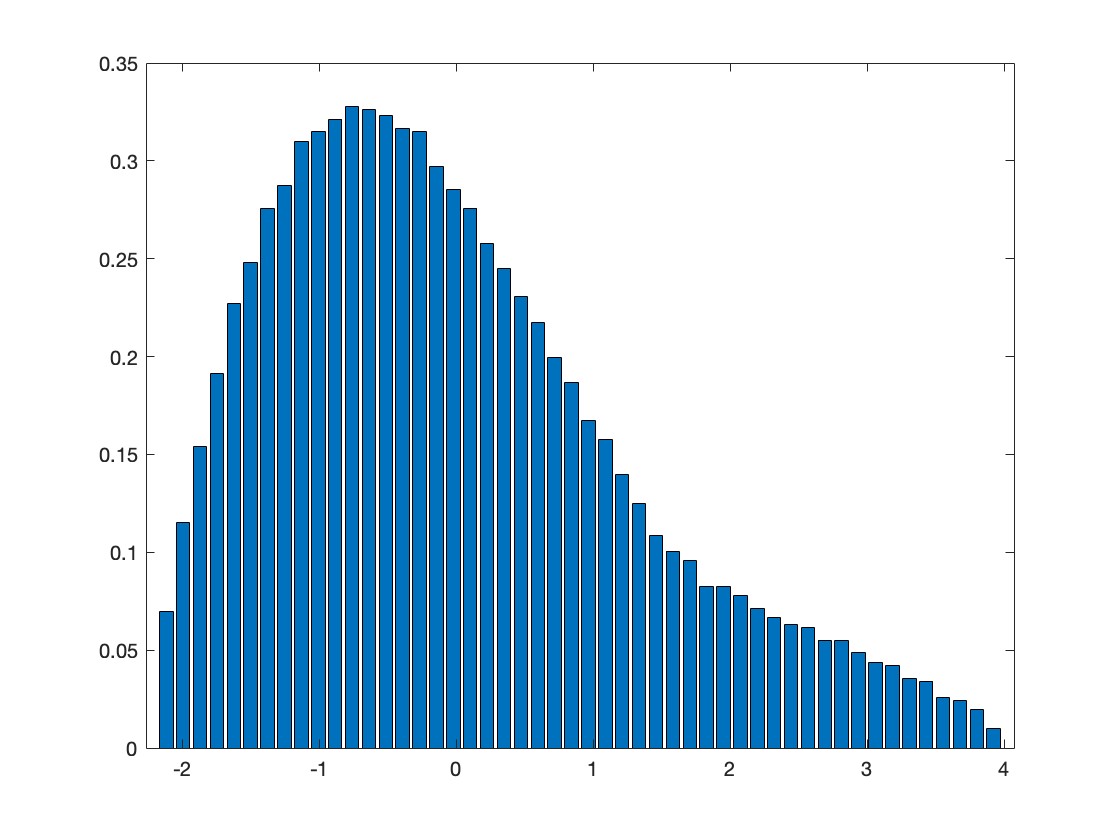}
        \caption{Eigenvalue distribution of $Y_N$, after applying ReLU entrywise to $X_N$}
    \end{subfigure}
    \begin{subfigure}{\textwidth}
         \includegraphics[width=12cm]{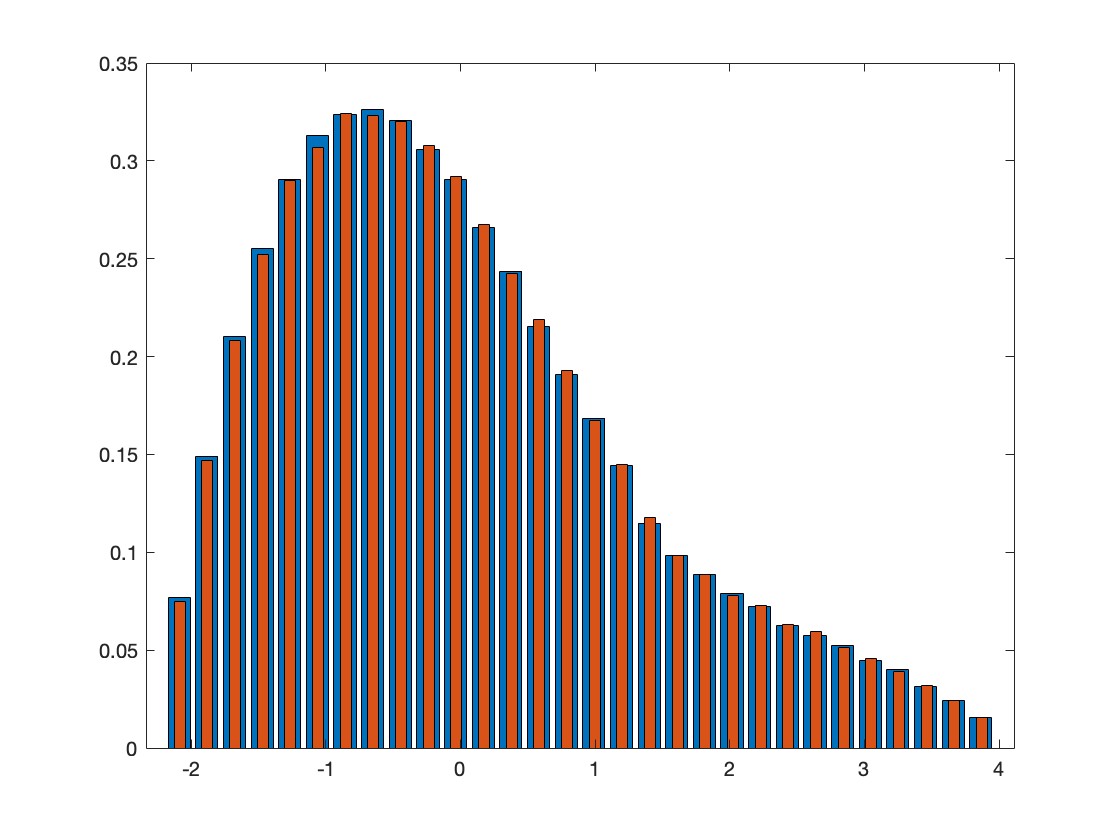}
    \caption{Superposition of the eigenvalues of the non-linear matrix $Y_N$ and of its Gaussian equivalent $\hat Y_N$}
    \end{subfigure}
    \caption{The effect of entrywise application of ReLU on $X_N$ from \Cref{eq:X_Neins} and comparison to the Gaussian equivalent $\hat Y_N$ from \Cref{eq:Y_Nhateins}; $N=5000$} 
\end{figure}

We apply now the ReLU function entrywise to $X_N$, resulting in the eigenvalue distribution for $Y_N$, as given in Figure 1(B).

In this example we have $\kk_2=5$, $\theta_1=1/2$, $\theta_2=\sqrt{5(1-2/\pi)}/2$, and Figure 1(C) superimposes the eigenvalue distribution of 
\begin{equation}\label{eq:Y_Nhateins}
    \hat Y_N:=\frac 12 X_N +\frac 12\sqrt{5(1-2/\pi)}Z_N
\end{equation}
to the preceding plot for $Y_N$. Both histograms agree perfectly.

\section{Extension to the multivariate case}
We can extend our investigations to functions of several independent matrices. Consider $l$ independent matrices $X_N^{(1)}=(x_{ij}^{(1)})_{i,j=1}^N, \dots , X_N^{(l)}=(x_{ij}^{(l)})_{i,j=1}^N$, where each of them is symmetric and orthogonally invariant. We build a new matrix 
$Y_N$ out of them by applying entrywise a function $f:\RR^l\to\RR$; that is, we
define $Y_N=(y_{ij})_{i,j=1}^N$ by
\begin{equation}\label{eq:Y-multivariate}
y_{ij}=\begin{cases}
\frac 1{\sqrt{N}} f(\sqrt{N} x^{(1)}_{ij}, \dots,\sqrt{N}x^{(l)}_{ij}), & i\not=j\\
0,& i=j.
\end{cases}
\end{equation}

By the same kind of arguments as before we get then the following description of the asymptotic free cumulants of $Y_N$.
\begin{theorem}\label{thm:cum-of-nl}
We have the following relation between the free cumulants $\kk^{(1)}_n, \dots, \kk^{(l)}_n$ of the asymptotic eigenvalue distributions of $X^{(1)}_N, \dots, X_N^{(l)}$, respectively, and the free cumulants $\kk_n^f$ of the asymptotic eigenvalue distribution of the random matrix $Y_N$:
\begin{enumerate}
\item
$\kk_1^f=0$;
\item
$\kk_2^f=c_2(f(g_1, \dots, g_l),f(g_1, \dots, g_l))=\EE[f(g_1, \dots, g_l)^2]-\EE[f(g_1, \dots, g_l)]^2$;
\item
for $n\geq 3$,
$$\kk_n^f= \sum_{j=1}^l \kk^{(j)}_n\cdot \EE[\partial_jf(g_1, \dots, g_l)]^n;$$
\end{enumerate}
where $g_1, \dots, g_l$ are independent Gaussian random variables with mean zero and variance $\kk^{(1)}_2, \dots, \kk^{(l)}_2$, respectively.
\end{theorem}

This yields the following multivariate version of a Gaussian equivalence principle, which we only state for the case $l=2$.

\begin{corollary}
Let $X^{(1)}_N$, $X^{(2)}_N$ and $Y_N$ be as above. Then the non-linear random matrix model $Y_N$ has the same asymptotic eigenvalue distribution as the linear model
$$\hat Y_N:=\theta_1 X^{(1)}_N+\theta_2 X^{(2)}_N+\theta Z_N,$$
where $Z_N$ is a symmetric standard \text{\rm GOE} random matrix which is independent from $X^{(1)}_N$ and $X_N^{(2)}$ and where
$$\theta_1:=\EE[\partial_1 f(g_1,g_2)],\qquad
\theta_2:=\EE[\partial_2 f(g_1,g_2)]$$
and
\begin{align*}
\theta:&= \sqrt{c_2(f(g_1,g_2),f(g_1,g_2))-\kk_2^{(1)}\theta_1^2-\kk_2^{(2)}\theta_2^2 },
\end{align*}
where $g_1$ and $g_2$ are independent Gaussian random variables with mean zero and variance $\kk^{(1)}_2$ and $\kk_2^{(2)}$, respectively. 
\end{corollary}

\begin{proof}
    Let us assume that our function $f$ is a (multivariate) polynomial. As we detailed before, by multilinearity, it suffices to consider monomial nonlinearities. To keep the following in a readable style, we will denote the $(i,j)$-entry of $X_N^{(k)}$ by $(x_{k})_{ij} $ .

    So, let monomials 
    \begin{multline*}
    f_1(x_1, \dots, x_l) = (x_1)^{m_1^{(1)}} (x_2)^{m_1^{(2)}} \cdots(x_l)^{m_1^{(l)}} , \dots, \\
    \dots
    f_n(x_1, \dots, x_l) = (x_1)^{m_n^{(1)}} (x_2)^{m_n^{(2)}} \cdots(x_l)^{m_n^{(l)}} 
    \end{multline*}
    with non-zero degree in each variable be given. If one degree is zero, then we can just replace $l$ by $l-1$. Assume again that we are given $n$ distinct indices $i_1, \dots, i_n$ and we want to know what $c_n(y_{i_1i_2}, y_{i_2i_3}, \dots, y_{i_ni_1} )$ looks like in leading order. Plugging in the definition of $Y_N$ we arrive at:
    \begin{multline*}
        c_n \left[ y_{i_1i_2}, y_{i_2i_3}, \dots, y_{i_ni_1} \right] \\= N^{(m - n)/2 }c_n\left[ (x_1)_{i_1i_2}^{m_1^{(1)}} (x_2)_{i_1i_2}^{m_1^{(2)}} \cdots (x_l)_{i_1i_2}^{m_1^{(l)}}, \dots,(x_1)_{i_ni_1}^{m_n^{(1)}}(x_2)_{i_ni_1}^{m_n^{(2)}} \cdots (x_l)_{i_ni_1}^{m_n^{(l)}}\right] 
    \end{multline*}

     Using the formula of Leonov and Shiryaev for $c_n$:
     \begin{align*}
         &c_n\left[ (x_1)_{i_1i_2}^{m_1^{(1)}}  (x_2)_{i_1i_2}^{m_1^{(2)}} \dots, \dots, \dots (x_{l-1})_{i_ni_1}^{m_n^{(l-1)}} (x_l)_{i_ni_1}^{m_n^{(l)}}\right] \\
        &=  \sum_{\pi \in \mathcal{P}(m), \pi \vee \sigma = 1_{m}} c_\pi \left[ (x_1)_{i_1i_2}, \dots, (x_1)_{i_1i_2}, (x_2)_{i_1i_2}, \dots, (x_2)_{i_1i_2}, \dots \right].
     \end{align*}

     The partition $\pi$ in the above sum must not, by independence, join any $(x_k)_{i_j i_{j+1}}$ with any of the $(x_h)_{i_r i_{r+1}}$ for $k \neq h$. Let us assume that $n \geq 3$. Since we are assuming that the indices $(i_1, i_2, \dots, i_n)$ form a full cycle, $\pi$ can only contain full cycles on each of the $X_N^{(i)}$ entries or connect them within one block. Let us now use our arguments from the proof of \Cref{thm:cumulants} in the multivariate setting. Let $s^{(k)}_j$ denote the number of blocks of $\pi$ of size $j$ which connect the $k$-variables withing one group; they will give a contribution of order $s^{(k)}_j(1 - j)$ and as such we can again argue that the total contribution  will be of order $ \sum_{j=1}^n s_j(1 - j )  = \sum_{j=1}^n s_j - m$, where $s_j = \sum_{k=1}^l s^{(k)}_j$. As in the single-matrix case we see that we need to maximize the number of blocks in order to obtain a leading order contribution, which is again $\mathcal{O}(N^{1-n})$. Now, since we need at least one full cycle on one of the $(k)$-indices the leading order contribution will be exactly the one with one full cycle on one of the entries of a matrix $X_N^{(k)}$ and the rest are pairings within each group. Note, however, that now we have the $(k)$-index as additional degree of freedom, so that several such partitions will make a leading order contribution. Namely, we can choose the index on which we will have a full cycle and then pair the remaining indices. We will thus have in total:
     \begin{align*}
         \kk_n^{f_1, \dots, f_n} &= \sum_{k} \kk_n^{(k)} m_1^{(k)} \kk_2^{(k)} \mathcal{P}_2(m_1^{(k)} -1) \cdot  m_2^{(k)} \kk_2^{(k)}\mathcal{P}_2(m_2^{(k)} -1) \cdots \\ &\qquad\qquad\qquad\qquad\cdots m_n^{(k)} \kk_2^{(k)}\mathcal{P}_2(m_n^{(k)} -1)  \times \prod_{j \neq k} \mathcal{P}_2(m_1^{(j)}) \cdots \mathcal{P}_2(m_n^{(j)}) \\
         &= \sum_{k} \kk_n^{(k)} m_1^{(k)}\kk_2^{(k)}\mathbb{E}[g_k^{m_1^{(k)} - 1}] \cdots m_n^{(k)}\kk_2^{(k)}\mathbb{E}[g_k^{m_n^{(k)} - 1}]  \prod_{j \neq k} \mathbb{E}[g_j^{m_1^{(k)}}] \cdots \mathbb{E}[g_j^{m_n^{(j)}}] \\
         &=  \sum_{k} \kk_n^{(k)} \prod_{j=1}^n \partial_k \mathbb{E}[f_j(g_1, \dots, g_n)],
     \end{align*}
     where the $g_k, k = 1, \dots, n$, are independent Gaussians with $g_k \sim \mathcal{N}(0, \kk_2^{(k)} )$. By multilinearity, we have the same identity for any multivariate polynomial.

     In the case $n = 2$ we have as in the univariate case that any pairing will contribute as long as at least one block connects the two different groups. Thus in the limit $N \rightarrow \infty$ we will obtain for multivariate monomials $f_1(x_1, \dots , x_l) = x_1^{m_1} \cdots x_l^{m_l}$ and $f_2(x) = x_1^{m'_1} \cdots x_l^{m'_l}$ that
     \begin{align*}
      \kk_2^{f_1, f_2} &= \prod_j (\kk_2^{(j)})^{(m_j + m_j')/2} \mathcal{P}_2( m_j +m_j') - \prod_j (\kk_2^{(j)})^{(m_j + m_j')/2} \mathcal{P}_2(m_j)  \times \mathcal{P}_2(m_j')\\
      &= \mathbb{E}[\prod_j g_j^{m_j+ m_j'} ] - \mathbb{E}[\prod_j g_j^{m_j}] \cdot \mathbb{E}[\prod_j g_j^{m_j'}] \\
      &= c_2(f_1(g_1,\dots,g_l), f_2(g_1,\dots,g_l)),
     \end{align*}
     where the $g_k, k= 1, \dots , l$, are independent Gaussians of variance $\kk_2^{(k)}$, respectively.
     
     In case that there are subcycles present we have a similiar situation as in the one matrix case since, by independence, the cumulants factorize into those of the cumulants of the individual matrices and we can transfer our arguments from the proof of \Cref{thm:cumulants}. If the index contains $d$ different cycles of length $d_j$ each, then, assuming without restriction that the degrees are large enough, we need each of these cycles once and now have the freedom to choose on which $X^{(k)}_N $ we choose the cycle; the important difference is that now only elements corresponding to the same matrix may be connected. Finally, we can connect the blocks that correspond to the same edge via pairings. Thus, again the order of a cumulant in this case can be at most  $N^{d - \sum_j d_j/2 - n/2}$; and contributions containing several disjoint cycles will also be subleading by the same argument of splitting blocks.

     Turning now to the computation of moments, note that in the present multivariate setting \Cref{lemma:cyclicycumulants} still holds and we can again decompose:
     \begin{align*}
        \sum_{i_1,\dots,i_n=1}^N c_\pi \left[y_{i_1i_2},y_{i_2i_3}, \dots, y_{i_ni_1} \right] &= \sum_{G} \sum_{\substack{i_1,\dots,i_n=1\\ G = G_i}}^N c_\pi \left[y_{i_1i_2},y_{i_2i_3}, \dots, y_{i_ni_1} \right].
    \end{align*}
    By the very same argument as in the univariate case the contributions of crossing partitions $\pi$ will vanish in the limit $N \rightarrow \infty$, since as we remarked, the contributions with several disjoint cycles being connected as well as those with subcycles have the same order as in the univariate case. And only contributions where $G$ is a cactus graph will survive, whose associated contribution will be $\kk_{\pi}^f$ for $\pi$ a noncrossing partition.
\end{proof}

\section{Example of the maximum function}

Again we expect that the multivariate results extend also beyond the case of polynomial functions. As an example, which might also be of some relevance in applications, let us check this numerically for the maximum function
$$f(x_1,x_2)=\max(x_1,x_2)=\frac 12 (x_1+x_2+\vert x_1-x_2\vert).$$
If we restrict to the symmetric situation $\kk_2^{(1)}=\kk_2^{(2)}=1$, then all our parameters are easy to calculate, namely
$$\EE[\max(g_1,g_2)]=\frac 1{\sqrt{\pi}},\qquad
\EE[\max(g_1,g_2)^2]=1$$
and thus
$$c_2(\max((g_1,g_2),\max(g_1,g_2))=1-\frac 1{\pi}$$
and furthermore
$$\theta_1=\theta_2=\text{prob}\{g_1>g_2\}=\frac 12;$$
thus
$$\theta=\sqrt{\frac 12 -\frac 1\pi}.$$
As our orthogonally invariant matrices we take
\begin{equation}\label{eq:XN1undX_2}
X_N^{(1)}=\frac{-A_N^2+B_N}{\sqrt{2}},\qquad
X_N^{(2)}=\frac{C_N^4+C_ND_N+D_NC_N}{\sqrt{12}},
\end{equation}
where $A_N,B_N,C_N,D_N$ are independent standard GOE. Figures 2(A) and 2(B) show their eigenvalue distributions, for one realization with $N=5000$.

\begin{figure}\label{fig:2}
    \centering
    \begin{subfigure}[t]{0.4\textwidth}
        \centering
        \includegraphics[width=5cm]{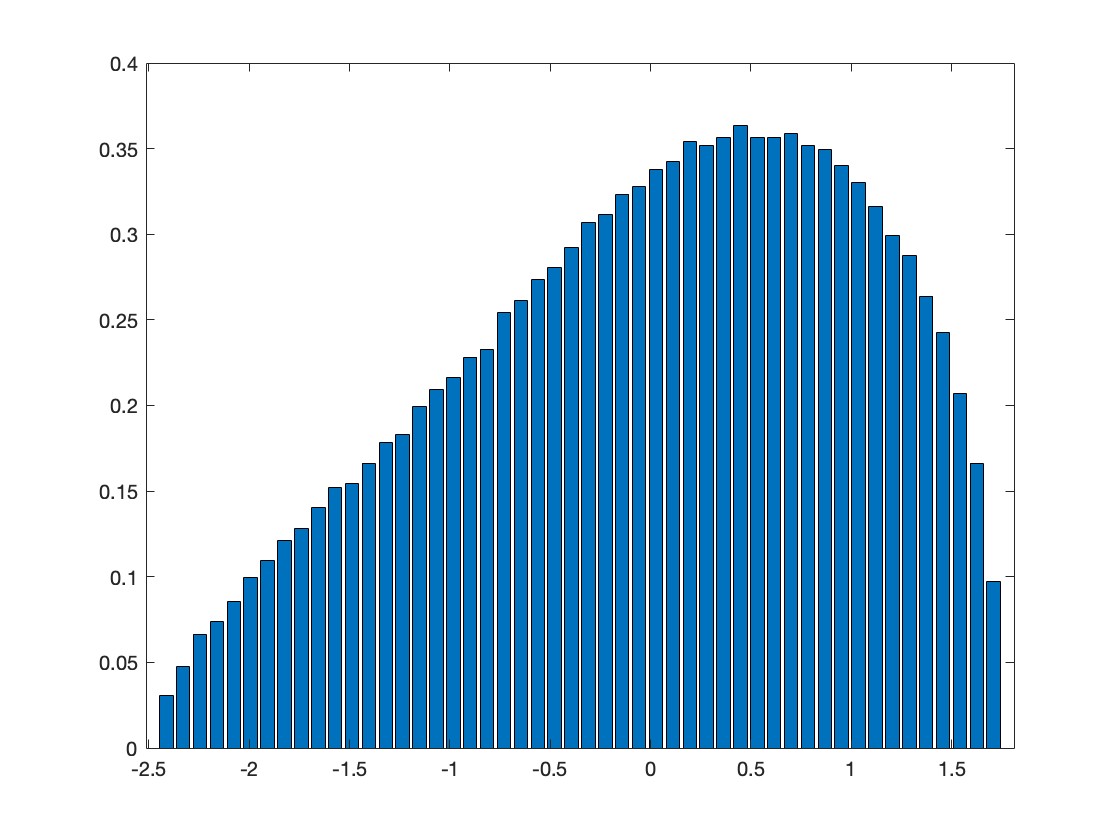}
        \caption{Eigenvalue distribution of $X_N^{(1)}$}
    \end{subfigure}%
    ~ 
    \begin{subfigure}[t]{0.4\textwidth}
        \centering
        \includegraphics[width=5cm]{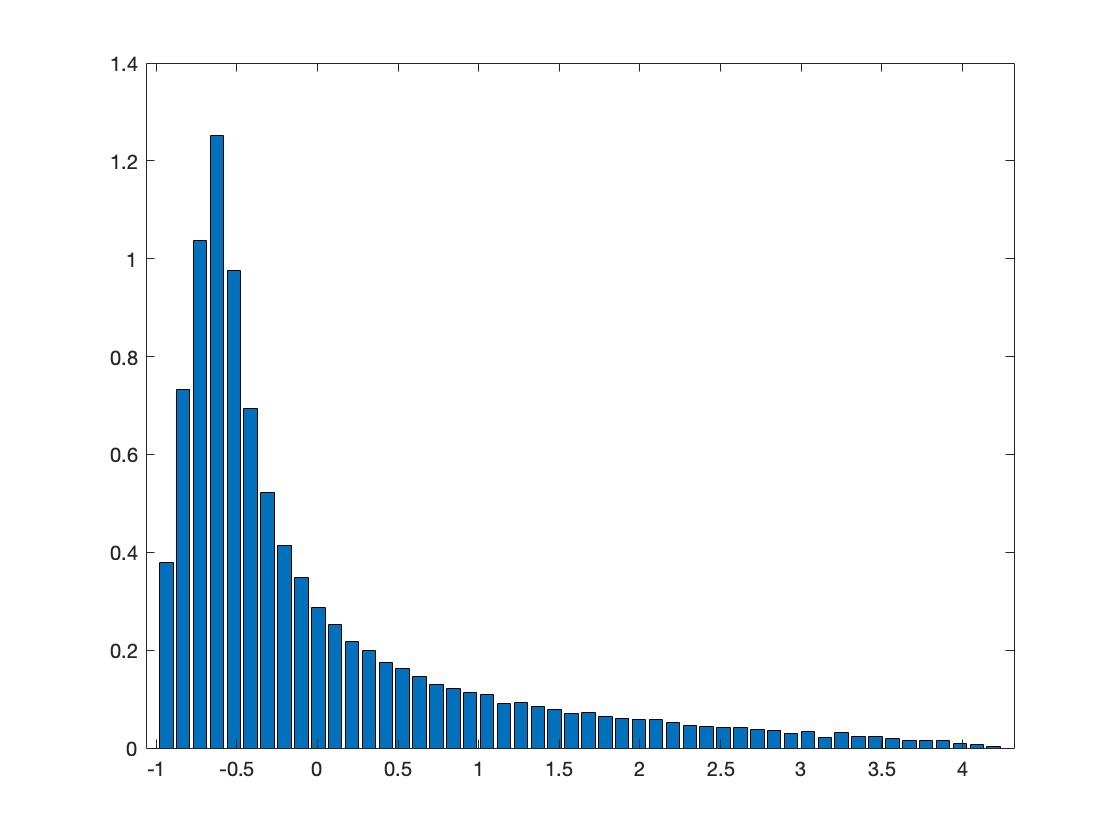}
        \caption{Eigenvalue distribution of $X_N^{(2)}$}
    \end{subfigure}
    \begin{subfigure}{\textwidth}
    \centering
         \includegraphics[width=8cm]{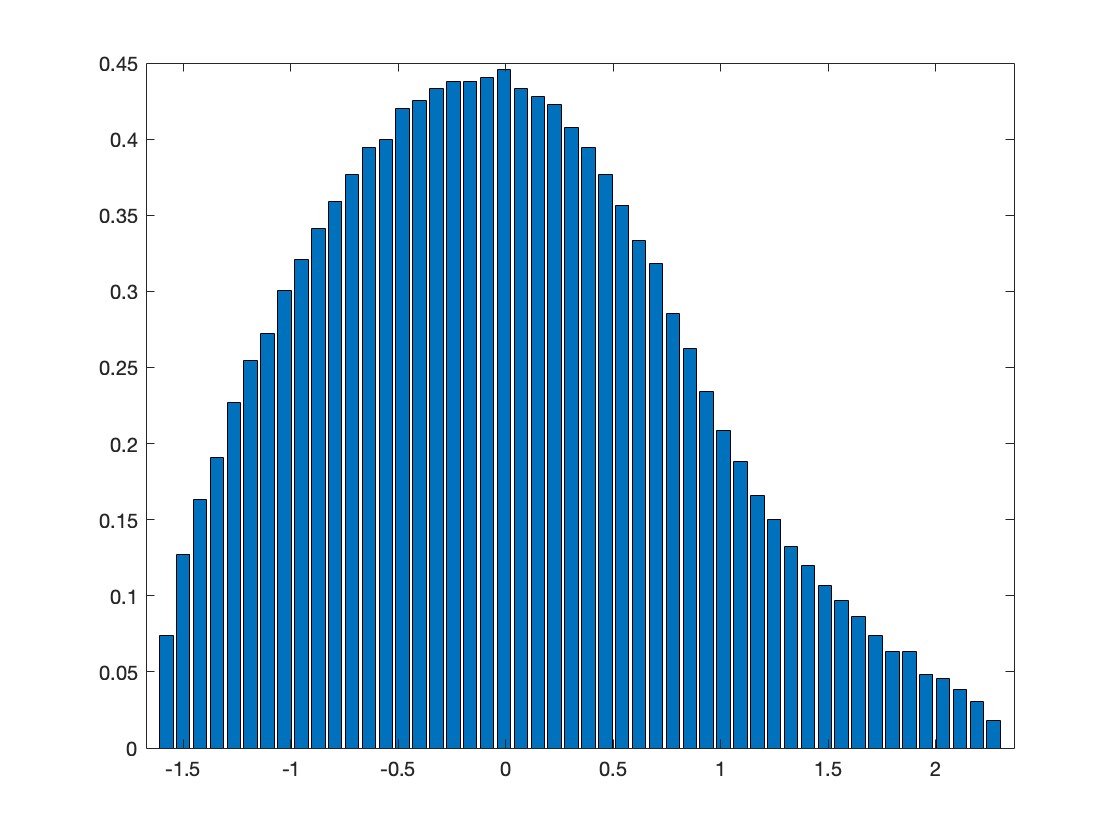}
    \caption{Eigenvalue distribution of $Y_N$, after applying max entrywise to $X_N^{(1)}$ and $X_N^{(2)}$}
    \end{subfigure}
    \caption{The effect of entrywise application of max on $X_N^{(1)}$ and $X_N^{(2)}$ from \Cref{eq:XN1undX_2}; $N=5000$} 
\end{figure}

We apply now the max-function entrywise to $X_N^{(1)}$ and $X_N^{(2)}$, which results in the eigenvalue distribution in Figure 2(C). (We should remark that there is one additional eigenvalue of order 40, which is outside the shown limits for the $x$-axis. This is of course irrelevant for the asymptotic eigenvalue distribution, but it will become important to deal with this when one considers questions of strong convergence and outliers. This eigenvalue comes from the non-vanishing mean of the entries of the matrix $Y_N$.)

\begin{figure}\label{fig:3}
    \centering
\includegraphics[width=12cm]{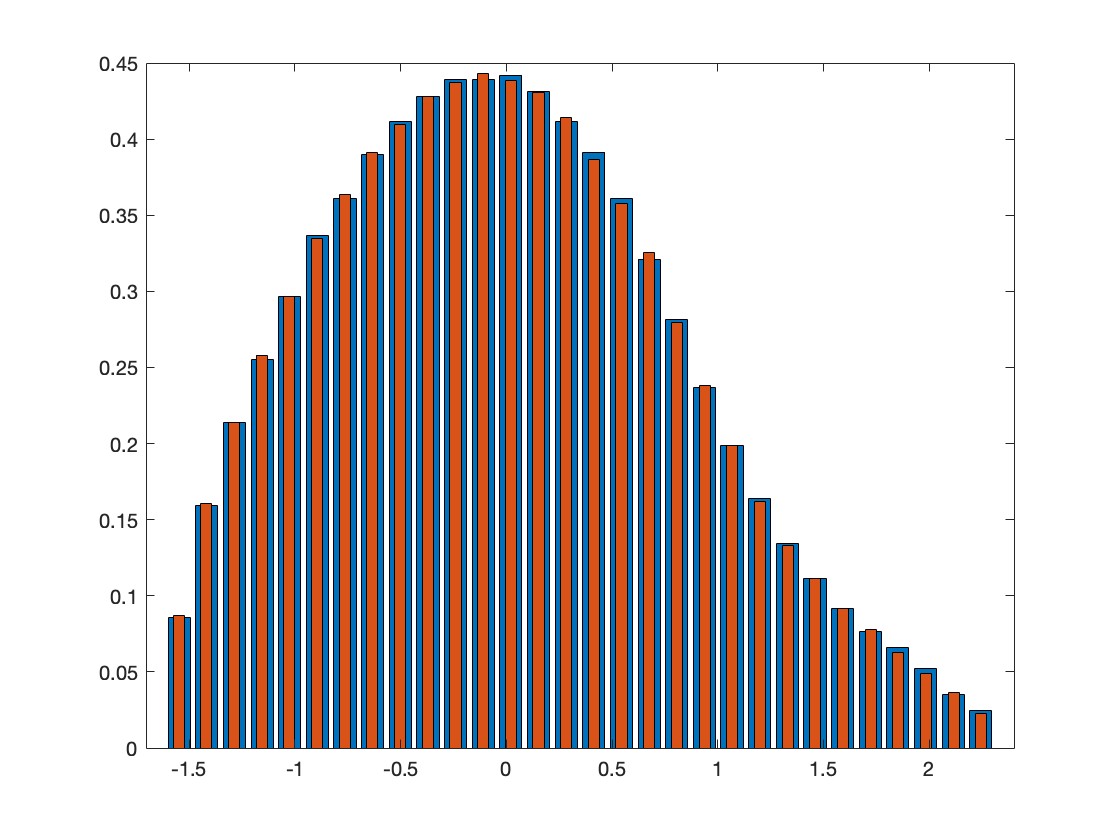}
\caption{
Superposition of the eigenvalue distributions of the non-linear matrix $Y_N$ and of its Gaussian equivalent $\hat Y_N$ from \Cref{eq:YNhatzwei}; $N=5000$}
\end{figure}

In Figure 3 we superimpose to this distribution the eigenvalue distribution of
\begin{equation}\label{eq:YNhatzwei}
\hat Y_N:=\frac 12 X_N^{(1)}+\frac 12 X_N^{(2)}+ \sqrt{\frac 12 -\frac 1\pi}Z_N,
\end{equation}
where $Z_N$ is a standard GOE, independent from $A_N,B_N,C_N,D_N$. Again, perfect agreement!

\section{Extension to the multivariate case with correlation}

We can extend our investigations also to a situation where the matrices $X^{(1)}_N$, \dots, $X^{(l)}_N$ are not independent, but have a correlation. For this we do not just assume the matrices to be separately orthogonally invariant, but we require the orthogonal invariance for the $l$-tuple $(X^{(1)}_N, \dots, X^{(l)}_N)$. This implies then that in addition to the individual free cumulants $\kk_n^{(1)}, \dots, \kk_n^{(l)}$ of $X^{(1)}_N, \dots, X^{(l)}_N$ we have also mixed free cumulants in the $l$ matrices, 
\begin{equation*}
\kappa_n^{(r_1,\dots,r_n)}:=\lim_{N\to\infty} N^{n-1}c_n(x_{i_1i_2}^{(r_1)},x_{i_2i_3}^{(r_2)},\dots,x_{i_ni_1}^{(r_n)}).
\end{equation*}
Of those mixed free cumulants only the second ones will be relevant for the effect of the non-linearities. 

\begin{theorem}\label{thm:cum-of-nl-corr}
We have the following relation between the free cumulants $\kk^{(r_1,\dots,r_n)}$ of the asymptotic eigenvalue distributions of a jointly orthogonally invariant family $X^{(1)}_N, \dots, X_N^{(l)}$, and the free cumulants $\kk_n^f$ of the asymptotic eigenvalue distribution of the non-linear random matrix $Y_N$ which is defined by \Cref{eq:Y-multivariate}.
\begin{enumerate}
\item
$\kk_1^f=0$;
\item
$\kk_2^f=c_2(f(g_1, \dots, g_l),f(g_1, \dots, g_l))=\EE[f(g_1, \dots, g_l)^2]-\EE[f(g_1, \dots, g_l)]^2$;
\item
for $n\geq 3$,
$$\kk_n^f
= \sum_{r_1, \dots , r_n = 1}^l \kappa_n^{({r_1}, \dots ,{r_n})} \prod_j^l  \mathbb{E} \left[ \partial_{r_j} f_j(g_1,\dots,g_l)\right];$$
\end{enumerate}
where $g_1, \dots, g_l$ are a Gaussian family of random variables with mean zero and the same covariance structure as the asymptotic covariance of $X^{(1)}_N, \dots, X^{(l)}_N$, that is,
$$c_2(g_r,g_s)=\kk_2^{(r,s)},\qquad \text{for all $1\leq r,s\leq l$.}$$
\end{theorem}

Again, this yields a corresponding Gaussian equivalence principle. As before, we only state it for $l=2$.

\begin{corollary}\label{thm:correlatedmatrices}
Let $X^{(1)}_N$ and $X^{(2)}_N$ be jointly orthogonally invariant random matrices such that a joint limit distribution of all orders exist. Let $Y_N$ be defined as in \Cref{eq:Y-multivariate}. Then the non-linear random matrix model $Y_N$ has the same asymptotic eigenvalue distribution as the linear model
$$\hat Y_N:=\theta_1 X^{(1)}_N+\theta_2 X^{(2)}_N+\theta Z_N,$$
where $Z_N$ is a symmetric standard \text{\rm GOE} random matrix which is independent from $X^{(1)}_N$ and $X_N^{(2)}$ and where
$$\theta_1:=\EE[\partial_1 f(g_1,g_2)],\qquad
\theta_2:=\EE[\partial_2 f(g_1,g_2)]$$
and
\begin{align*}
\theta:&= \sqrt{c_2(f(g_1,g_2),f(g_1,g_2))-\kk_2^{(1,1)}\theta_1^2-\kk_2^{(2,2)}\theta_2^2 -2\kk_2^{(1,2)}\theta_1\theta_2},
\end{align*}
where $g_1$ and $g_2$ are a Gaussian family of random variables with mean zero and the same covariance structure as the asymptotic covariance of $X^{(1)}_N$ and $X^{(2)}_N$, that is,
$$c_2(g_1,g_1)=\kk_2^{(1,1)},\qquad
c_2(g_2,g_2)=\kk_2^{(2,2)},\qquad
c_2(g_1,g_2)=\kk_2^{(1,2)}.
$$
\end{corollary}

\begin{example}
As an example, let us again consider the maximum function, but now we take
for our orthogonally invariant matrices
\begin{equation}\label{eq:XN1undXN2cor}
X_N^{(1)}=\frac{A_N^2-\sqrt{2}B_N}{\sqrt{3}},\qquad
X_N^{(2)}=\frac{A_N^4+C_NA_N+A_NC_N}{\sqrt{12}},
\end{equation}
where $A_N,B_N,C_N$ are independent standard GOE. We still have $\kk_2^{(1,1)}=1$ and $\kk_2^{(2,2)}=1$, but now we have a correlation $\kk_2^{(1,2)}=1/2$ between the two matrices, since $A_N$ appears in both of them. In this case our parameters are
$$\theta_1=\theta_2 =\frac 12,\qquad
\theta=\sqrt{\frac 14-\frac 1{2\pi}}.$$
 Figure 4 superimposes, for one realization with $N=5000$, the eigenvalue distribution of the entrywise applied max function of $X_N^{(1)}$ and $X_N^{(2)}$ to the eigenvalue distribution of 
\begin{equation}\label{eq:YNhatdrei}
\hat Y_N:=\frac 12 X_N^{(1)}+\frac 12 X_N^{(2)}+ \sqrt{\frac 14 -\frac 1{2\pi}}Z_N
\end{equation}
\end{example}

\begin{figure}\label{fig:4}
    \centering
\includegraphics[width=12cm]{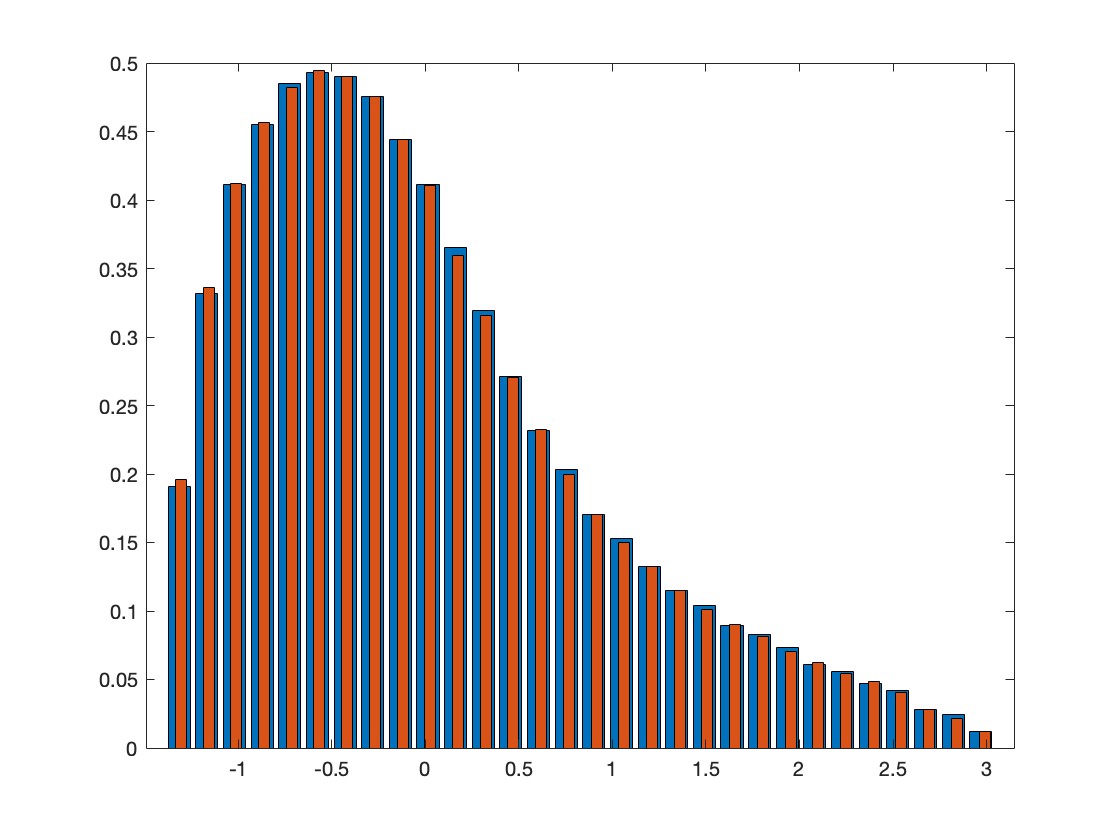}
\caption{
Superposition of the eigenvalues of the non-linear matrix $Y_N$, given as the max of the correlated matrices $X_N^{(1)}$ and $X_N^{(2)}$ from \Cref{eq:XN1undXN2cor}, and of its Gaussian equivalent $\hat Y_N$ from \Cref{eq:YNhatdrei}; $N=5000$}
\end{figure}

Let us now finally turn to giving the key arguments for the form of the cumulants in \Cref{thm:cum-of-nl-corr}.
\begin{proof}
    In case of a full cycle ($n \geq 3$) we know that the leading order is given by having a full cycle and the rest pairings, however, now the cycle may choose its elements arbitrarily from the correlated matrices, giving an asymptotic multiplicative contribution of $\kappa_n^{({r_1}, \dots ,{r_n})} \cdot N^{1-n}$ and the pairings may also pair arbitrarily within one group. Thus, in the limit $N \rightarrow \infty$, we will obtain the same as if we were computing the joint cumulants of the following random variables: Assume we have already chosen a full cycle with corresponding matrices $X_N^{(r_1)}, \dots, X_N^{(r_n)}$ and define a Gaussian family of random variables $h^{(r)}_{j}$ via having zero mean and covariance $c_2(h^{(r)}_{i}, h^{(s)}_{j}) = \delta_{ij}  \kk_2^{(r,s)}$. Then, we will obtain in the limit $N \rightarrow \infty$ the same as if we were computing the joint cumulant
    \begin{align*}
    &c_n((x_{r_1})_{i_1i_2}(h_{1}^{(r_1)})^{m_1^{(r_1)} -1} \prod_{s \neq r_1}  (h_{1}^{(s)})^{m_1^{(s)}},  \dots , (x_{r_n})_{i_ni_1}(h_{n}^{(r_n)})^{m_n^{(r_n)} -1} \prod_{s \neq r_n}  (h_{n}^{(s)})^{m_n^{(s)}}).
    \end{align*}
    And thus
    \begin{align*}
        \kk_n^{f_1, \dots , f_n} &= \sum_{r_1, \dots , r_n = 1}^l m_1^{(r_1)} \cdots m_n^{(r_n)} \kappa_n^{({r_1}, \dots ,{r_n})}  \mathbb{E} \left[ \prod_j \left( (h_{j}^{(r_j)})^{m_j^{(r_j)} -1} \prod_{s \neq r_j}  (h_{j}^{(s)})^{m_j^{(s)}}\right) \right] \\
        &= \sum_{r_1, \dots , r_n = 1}^l \kappa_n^{({r_1}, \dots ,{r_n})} \mathbb{E} \left[ \prod_j \partial_{r_j} f_j(h_j^{(1)},\dots,h_j^{(l)})\right] \\
        &= \sum_{r_1, \dots , r_n = 1}^l \kappa_n^{({r_1}, \dots ,{r_n})} \prod_j  \mathbb{E} \left[ \partial_{r_j} f_j(g_1,\dots,g_l)\right].
    \end{align*}
The last step follows from the fact that the $h_j^{(r)}$ are independent for different index $j$ and that for fixed $j$ they have the same covariance as $g_1,\dots,g_l$.
By multilinearity of the cumulants the results follow for general centered polynomials. 
    
    In the case $n=2$, we have to take into account that now, by the symmetry of our matrices, the pairings can also connect the two groups. Since there are now only pairings involved for the highest order, we get in this case
   directly 
   $$\kk_2^{f_1,f_2}=c_2(f_1(g_1, \dots, g_l),f_2(g_1, \dots, g_l)).$$

     If there are multiple cycles in the index, then we can again use the same arguments as in the single matrix/independent matrix case with the slight difference, that now we may also connect elements corresponding to different matrices, each giving a contribution $\kk_d^{(r_1, \dots, r_d )}$ for the subcycles and $\kk_2^{(r, s)}$ depending on which matrices are being paired but the maximal order will remain the same. The same also holds in case several disconnected cycles are being connected.
     
    The statement on the Gaussian equivalence principle follows directly from the concrete form of the free cumulants.
\end{proof}

\section{Concluding remark}
An interesting question for further investigation is whether we still have asymptotic freeness of our random matrices $Y_N$ from deterministic matrices. Note that this is one of the key features of orthogonally invariant random matrices. So it is plausible that this survives also under the entrywise non-linearities. However, the proof of this is not a direct consequence of our present results. If we restrict our setting to GOE, then applying a non-linearity makes the transition from the GOE to Wigner matrices; and our result reduces to the well-known fact that we still have a semicircular distribution in this case. One also knows that one still has asymptotic freeness of Wigner matrices from deterministic matrices, but showing this was technically more demanding than anticipated; for proofs see \cite{mingo2017free,anderson2010introduction}. We expect that following the same ideas as in \cite{mingo2017free}, relying on graph sum estimates, one should be able to prove the asymptotic freeness of our random matrices $Y_N$ from deterministic matrices.

\appendix

\section{Notations and Preliminaries}\label{section3}

\subsection{Some general notations}
For natural numbers $m,n\in\NN$ with $m<n$, we denote by $[m,n]$
the interval of natural numbers between $m$ and $n$, i.e.,
$$[m,n]:=\{m,m+1,m+2,\dots,n-1,n\};$$
furthermore, we put $[n]:=[1,n]$.

For a matrix $A=(a_{ij})_{i,j=1}^N$, we denote by $\Tr$ the
unnormalized and by $\tr$ the normalized trace,
$$\Tr(A):=\sum_{i=1}^N a_{ii},\qquad \tr(A):=\frac 1N \Tr(A).$$

We denote the set of permutations on $n$ elements by
$S_n$.

We say that $\pi=\{V_1,\dots,V_k\}$ is a \emph{partition} of a set $[1,n]$ if the sets $V_i$ are
disjoint and non--empty and their union is equal to $[1,n]$. We call $V_1,\dots,V_k$ the \emph{blocks}
of partition $\pi$. If all blocks of $\pi$ have size 2 then we call $\pi$ a \emph{pairing} or \emph{pair-partition}. We denote the set of all partitions of the set $[n]$ by $\cP(n)$ and the set of all pairings of the set $[n]$ (which is only non-empty if $n$ is even) by $\cP_2(n)$.

If $\pi=\{V_1,\dots,V_k\}$ and $\sigma=\{W_1,\dots,W_l\}$ are partitions of the same set, we say
that $\pi\leq \sigma$ if for every block $V_i$ there exists some block $W_j$ such that
$V_i\subseteq W_j$. For a pair of partitions $\pi,\sigma$ we denote by $\pi\vee \sigma$ the smallest
partition $\tau$ such that $\pi\leq \tau$ and $\sigma\leq \tau$. We denote by $1_n=\big\{ [n]
\big\}$ the biggest partition of the set $[n]$.

Let $i=(i_1,\dots,i_n)$ be a multi-index.
Then its \emph{kernel}, denoted by $\ker (i)$, is that partition in $\cP(n)$ whose blocks correspond exactly to the different values of the indices, that is
$$\text{$k$ and $l$ are in the same block of $\ker( i)$}\Longleftrightarrow i_k=i_l.
$$

\subsection{Classical cumulants}
Given some classical probability space $(\Omega,\ab P)$ we
denote by
$\EE$ the expectation with respect to the corresponding probability measure,
$$\EE(a):=\int_\Omega a(\omega)dP(\omega)$$
and by $L^{\infty-}(\Omega,P)$ the algebra of random variables for which all moments
exist.
Let us for the following put $\cA:=L^{\infty-}(\Omega,P)$.

We extend the linear functional $\EE:\cA\to\CC$
to a corresponding multiplicative functional on all
partitions by ($\pi\in \cP(n)$, $a_1,\dots,a_n\in\cA$)
\begin{equation}
\EE_\pi[a_1,\dots,a_n]:=\prod_{V\in\pi} \EE[a_1,\dots,a_n\vert_V],
\end{equation}
where we use the notation
$$\EE[a_1,\dots,a_n\vert_V]:=\EE(a_{i_1}\cdots a_{i_s})\qquad\text{for}\qquad
V=(i_1<\dots<i_s)\in\pi.$$

Then, for $\pi\in \cP(n)$, we define the \emph{classical cumulants} $\cc_\pi$ as multilinear
functionals on $\cA$ by
\begin{equation}
\cc_\pi[a_1,\dots,a_n]=\sum_{\substack{\sigma \in \cP(n)\\ \sigma\leq \pi}} \EE_\sigma[a_1,\dots,a_n]\cdot
\moeb_{}(\sigma,\pi),
\end{equation}
where $\moeb_{}$ denotes the M\"obius function on $\cP(n)$
(see \cite{rota1964foundations}).

The above definition is, by M\"obius inversion on $\cP(n)$, equivalent to
$$\EE(a_1\cdots a_n)=\sum_{\pi\in\cP(n)}\cc_\pi[a_1,\dots,a_n].$$
The $\cc_\pi$ are also multiplicative with respect to the blocks
of $\pi$ and thus determined by the values of
$$\cc_n(a_1,\dots,a_n):=\cc_{1_n}[a_1,\dots,a_n].$$

Note that we have in particular
$$\cc_1(a)=\EE(a)\qquad\text{and}\qquad
\cc_2(a_1,a_2)=\EE(a_1a_2)-\EE(a_1)\EE(a_2).$$

An important property of classical cumulants is the formula
of Leonov and Shiryaev \cite{leonov1959method} for cumulants with products as arguments; for the following formulation see \cite{mingo2017free}[Theorem~11.30].

\begin{theorem}[Formula of Leonov and Shiryaev] \label{Leonov}
    Consider a non-commutative probability space $(A, \varphi)$ and let $(c_\pi)_{\pi \in P}$ be the corresponding classical cumulants. Let $m, n \in \mathbb{N}$ and $1 \leq i(1) < i(2) < \cdots < i(m) = n$ be given and put
$$
\tau = \{(1, \ldots, i(1)), \ldots, (i(m - 1) + 1, \ldots, i(m))\}\in\cP(n).
$$
Consider now random variables $a_1, \ldots, a_n \in A$. Then we have
$$
c_m(a_1 \cdots a_{i(1)}, \ldots, a_{i(m-1)+1} \cdots a_{i(m)}) = \sum_{\substack{\pi \in P(n)\\ \pi \vee \tau = 1_n}} c_\pi[a_1, \ldots, a_n].
$$
\end{theorem}
Note that here $\vee$ denotes the join in the lattice of all partitions, see for example \cite{nica2006lectures} or \cite{mingo2017free} for more information on the lattice of (non-)crossing partitions.

The sum on the right-hand side is running over those partitions $\pi$ of $n$ elements
which satisfy $\pi\vee\tau=1_n$, which are, informally speaking, those
partitions which connect, together with $\tau$,
all the arguments of the cumulant $\cc_m$, when written in terms of the $a_i$.

Here is an example for this formula, for $\cc_2(a_1a_2,a_3a_4)$.
In order to reduce the number
of involved terms we will restrict to the special case where $\EE(a_i)=0$ (and thus also
$\cc_1(a_i)=0$) for all $i=1,2,3,4$.
There are three partitions $\pi\in\cP(4)$ without singletons which satisfy
$$\pi\vee \{(1,2),(3,4)\}=1_4,$$ namely
$$
\setlength{\unitlength}{0.3cm}
\begin{picture}(10,10)
  \thicklines
  \put(0,3){\line(0,1){2}}
  \put(0,3){\line(1,0){9}}
  \put(6,4){\line(0,1){1}}
  \put(3,4){\line(0,1){1}}
  \put(9,3){\line(0,1){2}}
\put(3,4){\line(1,0){3}} \thicklines
  \put(0,6){\line(0,1){2}}
  \put(0,6){\line(1,0){9}}
  \put(3,6){\line(0,1){2}}
  \put(6,6){\line(0,1){2}}
\put(9,6){\line(0,1){2}}
  \thicklines
  \put(0,1){\line(0,1){1}}
\put(0,1){\line(1,0){6}} \put(6,1){\line(0,1){1}}
  \put(3,0){\line(0,1){2}}
  \put(3,0){\line(1,0){6}}
  \put(9,0){\line(0,1){2}}
 \put(-.25,9.5){\dashbox{0.2}(4.2,1.5){}}
 \put(5.5,9.5){\dashbox{0.2}(4,1.5){}}
\put(0,10){$a_1$} \put(2.5,10){$a_2$} \put(5.75,10){$a_3$} \put(8,10){$a_4$}
\end{picture}
$$
and thus formula (\ref{Leonov}) gives in this case
\begin{multline*}
\cc_2(a_1a_2,a_3a_4)=\cc_4(a_1,a_2,a_3,a_4)+
\cc_2(a_1,a_4)\cc_2(a_2,a_3)+\cc_2(a_1,a_3)\cc_2(a_2,a_4).
\end{multline*}

\subsection{Haar distributed orthogonal random matrices and
            their Weingarten function}

In the following we will be interested in the asymptotics of
special matrix integrals over the group $\UU(N)$ of orthogonal
$N\times N$-matrices. We always equip the compact group
$\UU(N)$ with its Haar probability measure.
A random matrix whose distribution is this measure
will be called a \emph{Haar distributed orthogonal random
matrix}. Thus the expectation $\EE$ over this ensemble is
given by integrating with respect to the Haar measure.

The expectation of products of entries of Haar distributed
orthogonal random matrices can be described in terms of a
special function on pair partitions. Since such
considerations, in the unitary case, go back to Weingarten, Collins
\cite{collins2003moments} calls this function the \emph{Weingarten function} and
denotes it by $\Wg$. We will follow his notation.





We will only use the following formula to express general matrix integrals over the orthogonal group in terms of the Weingarten function. For all $n\in\NN$ and all $1\leq p_1,\dots,p_n,i_1,\dots,i_n\leq N$ we have
\begin{equation}\label{eq:Wick-unitary}
\EE\{u_{p_1 i_1} \cdots
u_{p_ni_n} \}
= \sum_{\pi,\sigma\in \cP_2(n)} \delta_{p,
\pi}\delta_{i,\sigma}\Wg(\pi,\sigma),
\end{equation}
where 
$$\delta_{p,\pi}=
\begin{cases}
1,&\text{if $\pi\geq \ker(p)$}\\
0,&\text{otherwise}
\end{cases}
$$
Note that this includes in particular the statement that all odd moments are zero (because $\cP_1(n)=\emptyset$ for $n$ odd).
This formula for the calculation of moments of the entries
of a Haar orthogonal random matrix bears some resemblance to the
Wick formula for the joint moments of the entries
of Gaussian random matrices. 
For more details on this Weingarten function, in particular also its asymptotic behaviour, we refer to
\cite{collins2006integration}.

\section{Correlation functions and cumulants for random matrices}
\label{sec:correlation}

\subsection{Correlation functions and partitioned pairings}
Let us consider $N\times N$-random matrices $B_1,\dots,B_n:\Omega\to M_N(\CC)$.
The
main information we are interested in are the ``correlation
functions" $\ff^{}_n$ of these matrices, given by classical
cumulants of their traces, i.e.,
$$\ff^{}_n(B_1,\dots,B_n):=\cc_n(\Tr(B_1),\dots,\Tr(B_n)).$$
Even though these correlation functions are cumulants, it is more adequate
to consider them as a kind of moments for our random matrices. Thus, we will
also call them sometimes \emph{correlation moments}.

We will also need to consider traces of products of the $B_i$ and their transposes, which are best
encoded via pair partitions. Thus, for $\pi\in\cP_2(2n)$,
$\ff^{}(\pi)[B_1,\dots,B_n]$ shall mean that we take cumulants of
traces of products along the cycles of $\hat\pi$. For $\pi\in\cP_2(2n)$ we denote here by $\hat\pi\in\cP(n)$ the partition in $\cP(n)$ that encodes which positions of the $B_i$ are connected by the blocks of $\pi$; more formally, $\hat\pi$ is the restriction of 
$$\pi\vee \{(1,2),(3,4),\dots,(2n-1,2n)\}$$
to the odd numbers $1,3,\dots,2n-1$. 
It is also useful to have an operation in the other direction, namely to lift an $\alpha\in\cP(n)$ to an $\check \alpha\in\cP(2n)$ by viewing $\alpha$ as a partition of the odd numbers $1,3,5,\dots,2n-1$ and $\check\alpha$ as a partition of the odd and even numers $1,2,3,4,\dots,2n-1,2n$, and then adding an even number to the block of $\alpha$ which contains the preceding odd number. Note that, for $\alpha\in\cP(n)$ and $\pi\in\cP_2(2n)$, the condition $\alpha\geq \hat \pi$ is equivalent to the condition $\check\alpha\geq \pi$. Furthermore, we have 
$$\check{\hat\pi}=\pi\vee \{(1,2),(3,4),\dots,(2n-1,2n)\}.$$
Applying the two operations in the other order does not make sense, since we have defined $\hat \pi$ only for pairings $\pi$.

If necessary, we can consider $\hat\pi$ also as a permutation in $S(n)$, since the pairs of $\pi$ give a cyclic structure on the blocks of $\hat\pi$ via matrix multiplication of the $B_i$ and their transposes. Note that $\pi\in \cP_2(2n)$ contains more information than $\hat\pi\in \cP(n)$, since the connection between a $B_l$ and a $B_k$ can happen by connecting the corresponding pairs of indices either by connecting two indices of the same parity or two indices of opposite parity. This information is relevant when we multiply the $B_i$ in the traces, since in the first case we must flip the second factor to its transpose. We prefer not to make this rigorous via a formal definition.
See \cite{mingo2013real}, or also the nice exposition \cite{collins2022weingarten} for this.

For an $n$-tuple
$B=(B_1,\dots,B_n)$ of random matrices and a cycle
$c=(i_1,i_2,\dots,i_k)$ of $\hat\pi$ with $k\leq n$ we denote by
$B\vert_c$ the product of the involved $B_i$ in the cycle order and with flips between transpose or not according to the information provided by $\pi$. 
Since we will apply the trace to such products, it does not matter were we start in our cycle, nor whether we start with a transpose or not. 

For any $\pi\in \cP_2(2n)$ and any $n$-tuple $B=(B_1,\dots,B_n)$ of random matrices we put then
$$\ff(\pi)[B_1,\dots,B_n]:=\ff_r(B\vert_{c_1},\dots,
B\vert_{c_r}),$$ where $c_1,\dots,c_r$ are the cycles of $\hat\pi$.

Example: Consider $\pi=\{(1,12),(2,10),(3,5),(4,6),(7,8),(9,11)\}\in \cP_2(12)$. Then $\hat\pi\in\cP_6$ is given by
$\{\{1,5,6\},\{2,3\},\{4\}\}$; and starting with the smallest element in each block will result in the corresponding cyclic structure $(1,5,6) (2,3)(4) \in S(6)$ and thus, by also taking into account that we need flips to the transposes by going from 1 to 5, from 5 to 6, and from 2 to 3,
\begin{align*}
\ff^{}((1,12),(2,10),(3,5),(4,6),(7,8),(9,11))[B_1,B_2,B_3,B_4,B_5,B_6]\\=
\ff^{}_3(B_1B_5^TB_6, B_2B_3^T,B_4)
=\cc_3(\Tr(B_1B_5^TB_6),\Tr(B_2B_3^T),\Tr(B_4)))
\end{align*}
Note that we could also start the first cycle at 5, but then we would get as cycle structure cycle $(5,1,6)$ with a flip by going from 5 to 1. However, the product $B\vert_c$ is $B_5B_1^TB_6^T$ gives under the  application of the trace the same as the product $B_1B_5^TB_6$, since
$$\Tr(B_5B_1^TB_6^T)=\Tr((B_6B_1B_5^T)^T)=\Tr(B_6B_1B_5^T)=\Tr(B_1B_5^TB_6).$$

Furthermore, we also need to consider more general products of such
$\ff(\pi)$'s. In order to index such products we will use pairs
$(\alpha,\pi)$ where $\pi$ is, as above, an element in $\cP_2(2n)$, and
$\alpha\in\cP(n)$ is a partition which is compatible with the cycle
structure of $\hat\pi$, i.e., $\alpha\geq\hat\pi$.

\begin{notation}
A \emph{partitioned pairing} is a pair $(\alpha,\pi)$ consisting of
$\pi\in \cP_2(2n)$ and $\alpha\in\cP(n)$ with $\alpha\geq\hat \pi$
(or equivalently, $\check\alpha\geq \pi$).

We will denote
the set of partitioned pairings of $n$ elements by $\cPP(n)$.
For such a $(\alpha,\pi)\in\cPP(n)$ we denote
$$\ff^{}(\alpha,\pi)[B_1,\dots,B_n]:=\prod_{V\in\cV}\ff(\pi\vert_V)
[B_1,\dots,B_n\vert_V].$$
\end{notation}

Example:
\begin{align*}
\ff\bigl(\{1,5\}\{2,3,4\},\{(1,10),(2,9),(3,5),(4,6),(7,8) \}\bigr)&[B_1,B_2,B_3,B_4,B_5]\\=
k_1(\Tr(B_1B_5))\cdot k_2(\Tr(B_2B_3^T),\Tr(B_4))
\end{align*}

Let us denote by $\Tr_\sigma$ for $\sigma \in\cP_2(2n)$ the product of traces along the blocks of $\sigma$. Then
we have the relation
$$\EE\{\Tr_\sigma[A_1,\dots,A_n]\}=
\sum_{\substack{\beta\in\cP(n)\\ \beta\geq\hat\sigma}} \ff(\beta,\sigma)[A_1,\dots,A_n].$$

By using the formula of Leonov and Shiryaev, \Cref{Leonov}, one
sees that in terms of the entries of our matrices $A_k=(a_{ij}^{(k)})_{i,j=1}^N$  our
$\ff(\beta,\sigma)$ can also be written as
\begin{equation}\label{eq:ff-entries}
\ff^{}(\beta,\sigma)[A_1,\dots,A_n]=\sum_{\substack{\alpha\leq\beta\\
\alpha\vee\hat\sigma=\beta}} \sum_{p_k,r_k=1}^N 
\delta_{[p,r],\sigma}
k_\alpha[a^{(1)}_{p_1r_1},\dots,
a^{(n)}_{p_nr_n}],
\end{equation}
where we denote with $[p,r]$ the index $2n$-tuple $(p_1,r_1,p_2,r_2,\dots,p_n,r_n)$

\subsection{Moments of orthogonally invariant random matrices}

For orthogonally invariant random matrices there exists a definite relation between
cumulants of traces and cumulants of entries.
We want to work out this connection
in this section.

\begin{definition}
Random matrices $A_1,\dots,A_n$ are called \emph{orthogonally
invariant} if the joint distribution of all their entries
does not change by global conjugation with any orthogonal
matrix, i.e., if, for any orthogonal matrix $O\in\UU(N)$, the
matrix-valued random variables $A_1,\dots, A_n:\Omega\to
M_N(\CC)$ have the same joint distribution as the
matrix-valued random variables $OA_1O^T, \dots, OA_nO^T :
\Omega \to M_N(\CC)$.
\end{definition}

Let $A_1,\dots,A_n$ be orthogonally invariant symmetric random matrices. We will now
try expressing the microscopic quantities ``cumulants of entries
of the $A_i$" in terms of the macroscopic quantities ``cumulants
of traces of products of the $A_i$".

In order to make this connection we have to use the orthogonal invariance of our
ensemble. By definition,
this
means that $A_1,\dots,A_n$ has the same distribution as
$\tilde A_1,\dots,\tilde A_n$ where $\tilde A_i:=OA_iO^T$.
Since this holds for any orthogonal $O$, the same is true after averaging over such $O$, i.e.,
we can take in the definition of the $\tilde A_i$ the $O$ as
Haar distributed
orthogonal random matrices, independent from $A_1,\dots,A_n$.
This reduces calculations for orthogonally invariant ensembles
essentially to properties of Haar orthogonal random matrices; in particular,
the Wick
formula for the $O$'s implies that we have an analogous Wick formula for joint moments in the
entries of the  $A_i$. Let us write
$A_k=(a_{ij}^{(k)})_{i,j=1}^N$ and $\tilde A_k=(\tilde a_{ij}^{(k)})_{i,j=1}^N$. Then we
can calculate:
\begin{align*}
\EE\bigl\{a_{p_1 r_{1}}^{(1)} \cdots\
         a_{p_{n} r_n}^{(n)}
\bigr\}&=\EE
\bigl\{\tilde a_{p_1 r_{1}}^{(1)} \cdots\
         \tilde a_{p_{n} r_n}^{(n)}
\bigr\}\\
&=\sum_{i,j} \EE\{u_{p_1 i_1}a^{(1)}_{i_1j_1}
{u_{r_1j_1}}\cdots u_{p_ni_n}a^{(n)}_{i_nj_n} {u_{r_nj_n}}\}\\
&=\sum_{i,j}\EE\{u_{p_1 i_1} {u_{r_1j_1}}\cdots
u_{p_ni_n} {u_{r_nj_n}}\}\cdot \EE\{a^{(1)}_{i_1j_1}\cdots
a^{(n)}_{i_nj_n}\}
\\
&= \sum_{i,j}\sum_{\pi,\sigma\in \cP_2(2n)} \delta_{[p,r],
\pi}\delta_{[i,j],\sigma}\Wg(\pi,\sigma) \cdot
\EE\{a^{(1)}_{i_1j_1}\cdots a^{(n)}_{i_nj_n}\},
\end{align*}
where we denote with $[p,r]$ the index $2n$-tuple $(p_1,r_1,p_2,r_2,\dots,p_n,r_n)$. Then we can write
$$\EE\bigl\{a_{p_1 r_{1}}^{(1)} \cdots\
         a_{p_{n} r_n}^{(n)}\bigr\}
=\sum_{\pi\in\cP_2(2n)}\delta_{[p,r],\pi}\cdot \GG(\pi)[A_1,\dots,A_n],
$$
where
\begin{align}
\label{eq:definitionofGG}
\GG(\pi)[A_1,\dots,A_n]:&= \sum_{\sigma\in\cP_2(2n)}
\Wg(\pi,\sigma) \cdot\sum_{i,j}\delta_{[i,j],\sigma}
\EE\{a^{(1)}_{i_1j_1}\cdots a^{(n)}_{i_nj_n}\}.
\end{align}
Recall that we have
$$\sum_{i,j}\delta_{[i,j],\sigma}
a^{(1)}_{i_1j_1}\cdots a^{(n)}_{i_nj_n}=\sum_ {\substack{i,j \\ \ker([i,j])\geq \sigma }}
a^{(1)}_{i_1j_1}\cdots a^{(n)}_{i_nj_n}=
\Tr_{\sigma}[A_1,\dots,A_n].$$

So we have 
\begin{align}
\GG(\pi)[A_1,\dots,A_n]&=\sum_{\sigma\in\cP_2(2n)} \Wg(\pi,\sigma) \cdot \EE\{\Tr_{\sigma}[A_1,\dots,A_n]\}.\notag\\
&=\sum_{\sigma\in\cP_2(2n)} \Wg(\pi,\sigma)\cdot\sum_{\substack{\alpha\in\cP(n)\\ \alpha\geq\hat\sigma}}
\ff(\alpha,\sigma)[A_1,\dots,A_n]\notag\\
&=\sum_{(\alpha,\sigma)\in\cPP(n)}\Wg(\pi,\sigma)\cdot \ff(\alpha,\sigma)[A_1,\dots,A_n].\notag
\end{align}
The important point here is that $\GG(\pi)[A_1,\dots,A_n]$ depends only on the macroscopic
correlation moments of $A$.

We can extend the above to products of expectations by
\begin{align*}
\EE_\beta [a_{p_1 r_{1}},\dots,
        a_{p_{n} r_n}]=\sum_{\substack{\pi\in \cP_2(2n)\\ \hat\pi\leq \beta}}
\delta_{[p,r],\pi}\cdot \GG(\beta,\pi)[A_1,\dots,A_n]\qquad\text{for $\beta\in \cP(n)$},
\end{align*}
where $\GG(\beta,\pi)$ is given by
multiplicative extension:
\begin{align}
\label{eq:definitionofGGmulti}\GG(\beta,\pi)[A_1,\dots,A_n]:&=\prod_{V\in
\beta}\GG(\pi\vert_V)[A_1,\dots,A_n\vert_V]\notag\\
&=\sum_{\substack{(\alpha,\sigma)\in\cPP(n)\\ \alpha\leq\beta}} \Wg_\beta(\pi,\sigma)\cdot
\ff(\alpha,\sigma)[A_1,\dots,A_n].
\end{align}
We have here also used the multiplicative extension of the Weingarten function: for $\beta\geq\hat \pi\vee\hat\sigma$, we put
$$\Wg_\beta(\pi,\sigma):=\prod_{V\in\check\beta}\Wg(\pi|_{V},\sigma|_{V}).$$

Now we can look on the cumulants of the entries of our unitarily invariant matrices $A_i$; they
are, for $\alpha\in\cP(n)$, given by
\begin{align*}k_\alpha\bigl\{a^{(1)}_{p_1 r_{1}},\dots,a^{(n)}_{p_{n} r_n}\}&=
\sum_{\substack{{\beta}\in\cP(n)\\ {\beta}\leq \alpha}} \moeb_{}(\beta,\alpha)\cdot\EE_{\beta} [a^{(1)}_{p_1
r_{1}},\dots,
        a^{(n)}_{p_{n} r_n}]\\
&=\sum_{{\beta}\leq \alpha}\sum_{\substack{\pi\in \cP_2(2n)\\ \hat\pi\leq \beta}}
\delta_{[p,r],\pi}\moeb_{}(\beta,\alpha)\cdot \GG(\beta,\pi)[A_1,\dots,A_n]\\
&=\sum_{\substack{\pi\in \cP_2(2n)\\ \hat\pi\leq \alpha}}
\delta_{[p,r],\pi}\sum_{\substack{\beta \in\cP(n)\\ \alpha\geq
{\beta}\geq\hat\pi}}\moeb_{}(\beta,\alpha)\cdot \GG(\beta,\pi)[A_1,\dots,A_n].
\end{align*}
With the definition
\begin{equation}\label{eq:def-KK}
\KK(\alpha,\pi)[A_1,\dots,A_n]:=\sum_{\substack{\beta \in\cP(n)\\ \alpha\geq
{\beta}\geq\hat\pi}}
\moeb_{}(\beta,\alpha)\cdot \GG(\beta,\pi)[A_1,\dots,A_n]
\end{equation}
we thereby get
\begin{equation}\label{eq:k-KK}
k_\alpha\bigl\{a^{(1)}_{p_1 r_{1}},\dots,a^{(n)}_{p_{n} r_n}\}=
\sum_{\substack{\pi\in \cP_2(2n)\\ \hat\pi\leq \alpha}}
\delta_{[p,r],\pi}\cdot\KK(\alpha,\pi)[A_1,\dots,A_n];
\end{equation}
in particular
\begin{equation}\label{eq:k-KK2}
k_n\bigl\{a^{(1)}_{p_1 r_{1}},\dots,a^{(n)}_{p_{n} r_n}\}=
\sum_{\pi\in \cP_2(2n)}
\delta_{[p,r],\pi}\cdot\KK(1_n,\pi)[A_1,\dots,A_n];
\end{equation}
It
follows that
\begin{align*}
\ff^{}(\beta,\sigma)[A_1,\dots,A_n]&=\sum_{\substack{\alpha\leq\beta\\
\alpha\vee\hat\sigma=\beta}} \sum_{p,r} 
\delta_{[p,r],\sigma}
k_\alpha[a^{(1)}_{p_1r_1},\dots,
a^{(n)}_{p_nr_n}]\\
&=   \sum_{\substack{\alpha\leq\beta\\
\alpha\vee\hat\sigma=\beta}} \sum_{p,r} \delta_{[p,r],\sigma}
\sum_{\substack{\pi\in \cP_2(2n)\\ \hat\pi\leq \alpha}}
\delta_{[p,r],\pi}\cdot\KK(\alpha,\pi)[A_1,\dots,A_n]\\
&=\sum_{\substack{\alpha\leq\beta\\
\alpha\vee\hat\sigma=\beta}} \sum_{\substack{\pi\in \cP_2(2n)\\ \hat\pi\leq \alpha}} \KK(\alpha,\pi)[A_1,\dots,A_n]\cdot N^{\#(\sigma\vee\pi)}
\end{align*}
and thus

\begin{equation}\label{eq:ff-kk}
\ff^{}(\beta,\sigma)[A_1,\dots,A_n]
=\sum_{\substack{(\alpha,\pi)\in\cPP(n)\\
\alpha\vee\hat\sigma=\beta}}\KK(\alpha,\pi)[A_1,\dots,A_n] \cdot N^{\#(\sigma\vee\pi)}.
\end{equation}

\begin{remark}

1) Note that although the quantity $\KK$ is defined by \Cref{eq:def-KK} in terms of the
macroscopic moments of the $A_i$, they have also a very concrete meaning in terms of cumulants
of entries of the $A_i$. Namely, if we choose $\pi \in \cP_2(2n)$ and $1\leq p_1,q_1,\dots,p_n,q_n\leq N$ such that each index appears exactly twice, then
\Cref{eq:k-KK} becomes, when we set $\alpha = 1_n$,
\begin{equation}
\KK(1_n,\pi)[A_1,\dots,A_n]=
k_n\bigl\{a^{(1)}_{p_1 r_{1}},\dots,a^{(n)}_{p_{n} r_n}\}\qquad
\text{if $\ker([p,q])=\pi$}
\end{equation}
as the only term in the sum that survives is the one for
$\pi$.


2) \Cref{eq:ff-kk} should be considered as
a kind of moment-cumulant formula in our context, thus it should
contain all information for defining the ``cumulants"
$\KK$ in terms of the moments $\ff$. Actually, we can solve this
linear system of equations for
$\KK$ in terms of $\ff$, by using 
\Cref{eq:def-KK} to define $\KK$ and
\Cref{eq:definitionofGGmulti} for $\GG$.
\begin{align}
\notag \KK(\alpha,\pi)[A_1,\dots,A_n]&=\sum_{\substack{\beta \in\cP(n)\\  \alpha\geq
{\beta}\geq\hat\pi}}\moeb_{}(\beta,\alpha)\cdot \GG(\beta,\pi)[A_1,\dots,A_n]\\ \notag
&=\sum_{\substack{\beta \in\cP(n)\\ \alpha\geq
{\beta}\geq\hat\pi}}\moeb_{}(\beta,\alpha)\cdot\sum_{\substack{(\gamma,\sigma)\in\cPP(n)\\ \gamma\leq\beta}} \Wg_\beta(\pi,\sigma)\cdot
\ff(\gamma,\sigma)[A_1,\dots,A_n]\\
&=\sum_{(\gamma,\sigma)\in\cPP(n)}\ff(\gamma,\sigma)[A_1,\dots,A_n]\sum_{\substack{\beta \in\cP(n)\\ \alpha\geq
{\beta}\geq\hat\pi\vee\gamma}}\moeb_{}(\beta,\alpha)\cdot \Wg_\beta(\pi,\sigma) \label{eq:kk-rel-Weingarten}
\end{align}

\end{remark}

\subsection{Asymptotic behavior of correlation moments and cumulants for orthogonally invariant random matrices}
Our main interest lies now in random matrices which are orthogonally invariant and which have limit distributions of all orders; this means by definition that all limits as in \Cref{eq:limitsofcorrelations} exist. From this it follows that the leading order of $\ff(\pi)=\ff(1_n,\pi)$ is given by $N^{2-\#(\hat\pi)}$ and, more generally, the leading order of $\ff(\alpha,\pi)$ is given by 
$N^{2\#(\alpha)-\#(\hat\pi)}$. 
From \Cref{eq:ff-kk} one can deduce that the leading order of $\kk(\alpha,\pi)$, for $(\alpha,\pi)\in\cPP(n)$, is given by the term $(\beta,\sigma)=(\alpha,\pi)$ and thus must be of order
$N^{2\#(\alpha)-\#(\hat\pi)-n}$.
(Indeed, this should also follows from \Cref{eq:kk-rel-Weingarten} and the leading order of the relative cumulant of the Weingarten function, as given in \cite{collins2006integration}.) \Cref{eq:k-KK2} yields then that we need cycles in the index structure to get non-vanishing cumulants of the entries of our orthogonally invariant matrices, as spelled out in  \Cref{lemma:orthcumulants}.

\section{Cumulants with multiple disjoint cycles}\label{section4}

Let us here prove the statement on cumulants of entries whose index structure consists of several disjoint cycles. First, we need some definitions to bound the order of such cumulants:

\begin{definition}
    Let $i = (i_1, \dots, i_n)$ be a multiindex consisting of $r$ many cycles of lengths $l_1, \dots, l_r$:
    \begin{multline*}
    j_1 = i_2, \ j_2  = i_3, \  \dots,\  j_{l_1} = i_1,\qquad 
    j_{l_1+1}=i_{l_1 +2},\ \dots,\ j_{l_1+l_2}=i_{l_1+1},\quad \dots\\
    \\ \dots\quad 
    j_{n-l_r+1}=i_{n-l_r+2},\ \dots,\ j_n=i_{n-l_r+1}
    \end{multline*} and let monomials $f_1(x) = x^{m_1}, \dots, f_n(x) = x^{m_n} $ be given. Then, we say that a block $I \in \pi$ of partition $\pi \in \mathcal{P}(m_1 + \cdots + m_n)$ is an intercycle-block if it connects indices belonging to the disjoint cycles in $i$. If a block is not an intercycle-block, then we call it an intra-cycle-block.
\end{definition}

    Assume now that $i$ consists of disjoint cycles. Let monomials $f_1(x) = x^{m_1}$, \dots, $f_n(x) = x^{m_n} $ be given.  Then, we want to investigate the order of the joint cumulant $c_n(f_1(x_{i_1j_1}), \dots , f_n(x_{i_nj_n}))$.
    
    If the multiindex consists of several disjoint cycles, then we may adapt the definition of $G_i$ to consist of the disjoint union of the graphs associated to its  constituent cycles. Let us first assume that we only have several disconnected cycles, without subcycles. Let $r$ be the number of disjoint cycles and $l_1,\dots,l_r$ the length of those cycles. Let us denote the inter-cycle blocks by $I_1,\dots,I_q$.
    If the inter-cycle block $I_j$ connects $f_j$ many of the disconnected cycles, then the multiplicative contribution associated to this block will be of order $N^{2- f_j - \#I_j} $, where $\#I_j$ is the size of the block $I_j$.  Applying now \Cref{Leonov} and  \Cref{lemma:orthcumulants}, we see that $c_\pi$ will be of order $N^t$ with
    \begin{align*}
        t&= \sum_{j=1}^q(2 - f_j - \#I_j) + \sum_{j = 2}^{\max \{ l_i\} } s_j(1 - j) = \sum_{j=1}^q (2 - f_j) + {\sum_{j=2}^{\max \{ l_i\} } s_j - m};
    \end{align*}
    as before, $s_j$ is here the number of intracycle blocks of size $j$.
    Now, by the definition of intercycle-blocks we have $f_j \geq 2$. We conclude that $t$ is maximal if $q=1$, $f_1 = r$ and if we have the maximal number of blocks. We can  conclude that the leading order contribution is given by a partition $\pi$ which contains one large block containing each cycle once and apart from this we only have blocks which pair an argument with itself. We infer that the order might be at most (including the $N^{(m-n)/2}$ prefactor) $N^t$ with
    \begin{align*}
        t = 2- r + \frac{m - \sum_i l_i}{2} - m + \frac{m-n}{2} = 2 - r - \frac 12\sum_i l_i - \frac{n}{2} = 2 - r - n.
    \end{align*}
    
    Thus, for disjoint cycles without subcycles the order for the $Y_N$ is the same as for the $X_N$.
    If the cycles, however, have subcycles, then as for the case of one cycle, the order might increase.
    If cycles show up several times, then in order to satisfy $\pi \vee \tau$, we need to connect elements from the different cycles, which we can achieve by choosing one subcycle from each cycle and then choosing the remaining blocks within each cycle as we did above. As before, this increase in the order will be compensated by the fact that there are less index-tuples which correspond to such a subcycle situation. Hence the contribution of such subfactor situations will be negligible in the limit.

\bibliographystyle{amsalpha}
\bibliography{non-linear}

\end{document}